\documentclass[11pt]{article}
\pdfoutput=1
\usepackage{amsmath, amsfonts, amsthm, amssymb, latexsym, mathrsfs, pifont,
tabu, enumitem}
\usepackage[cm]{fullpage}
\usepackage[capitalise]{cleveref}
\usepackage[ruled, linesnumbered]{algorithm2e}
\usepackage{makecell}
\usepackage{adjustbox}

\usepackage{url}

\newtheorem{thm}{Theorem}[subsection]
\newtheorem{prop}[thm]{Proposition}
\newtheorem*{prop*}{Proposition}
\newtheorem{cor}[thm]{Corollary}
\newtheorem{lemma}[thm]{Lemma}
\newtheorem{defi}[thm]{Definition}
\newtheorem{exa}[thm]{Example}
\newtheorem{algo}[thm]{Algorithm}

\newenvironment{de}[1][]{\begin{defi}[#1]\rm}{\end{defi}}
\newenvironment{ex}{\begin{exa}\rm}{\end{exa}}
\newenvironment{alg}{\begin{algo}\rm}{\end{algo}}

\newcommand{\defn}[1]{\textbf{\textit{#1}}}
\numberwithin{equation}{section}

\newcommand{\set}[2]{\ensuremath{\{#1 : #2 \}}}
\newcommand{\genset}[1]{\ensuremath{\langle\: #1 \:\rangle}}
\renewcommand{\to}{\longrightarrow}

\DeclareMathOperator{\im}{im}

\newcommand{\num}{\mathbf{num}}
\newcommand{\vect}{\mathbf{vec}}

\SetKwProg{Fn}{Function}{}{}
\SetKwFunction{FZeroIfNotSubset}{ZeroIfNotSubset}
\SetKwFunction{FRemoveDuplicateRows}{RemoveDuplicateRows}
\SetKwComment{Comment}{}{}
\SetKwInOut{Input}{Input}\SetKwInOut{Output}{Output}
\DontPrintSemicolon

\newcommand{\B}{\mathbb{B}}
\renewcommand{\S}{\mathbb{S}}
\newcommand{\Bn}{M_n(\B)}
\newcommand{\Bm}[1]{M_{#1}(\B)}
\newcommand{\Bmn}{M_{m,n}(\B)}
\newcommand{\Refln}{M_n^{\text{id}}(\B)}
\newcommand{\Refl}[1]{M_{#1}^{\text{id}}(\B)}
\newcommand{\Halln}{M_n^{\text{S}}(\B)}

\newcommand{\MTn}{M_n^{\mathcal{T}}(\B)}

\newcommand{\UTn}{UT_n(\B)}
\newcommand{\LTn}{LT_n(\B)}
\renewcommand{\L}{\mathscr{L}}
\newcommand{\R}{\mathscr{R}}
\newcommand{\D}{\mathscr{D}}
\newcommand{\J}{\mathscr{J}}
\renewcommand{\H}{\mathscr{H}}
\newcommand{\N}{\mathbb{N}}
\newcommand{\Np}{\N^{+}}
\newcommand{\Z}{\mathbb{Z}}

\renewcommand{\d}{\mathbf{d}}

\newcommand{\K}{\mathbb{K}}
\newcommand{\Kmin}{\K^{\infty}}
\newcommand{\Kmint}{\K^{\infty}_t}
\newcommand{\Kmax}{\K^{-\infty}}
\newcommand{\Kmaxt}{\K^{-\infty}_t}

\newcommand{\BGSet}{\mathcal{B}_{n,n}}

\newcommand{\AND}{\qquad\text{and}\qquad}
\newcommand{\mat}[4]{\begin{pmatrix}#1&#2\\#3&#4\end{pmatrix}}

\newcommand{\COMMA}{,\quad}

\newcommand{\RowS}{\Lambda}
\newcommand{\RowB}{\lambda}
\newcommand{\ColS}{P}
\newcommand{\ColB}{\rho}

\usepackage[backend=biber, bibencoding=utf8, giveninits=true]{biblatex}

\bibliography{matrix}

\title{Minimal generating sets for matrix monoids}
\author{F. Hivert, J. D. Mitchell, F. L. Smith, and W. A. Wilson}
\date{\today}
\begin{document}

\maketitle

\begin{abstract}
  In this paper, we determine minimal generating sets for several well-known monoids of 
  matrices over certain semirings. In particular, 
  we find minimal generating sets for the monoids consisting of: all $n\times n$
  boolean matrices when $n\leq 8$; the $n\times n$ boolean matrices containing
  the identity matrix (the \textit{reflexive} boolean matrices) when $n\leq 7$;
  the $n\times n$ boolean matrices containing a permutation (the \textit{Hall}
  matrices) when $n \leq 8$; the upper, and lower, triangular boolean matrices
  of every dimension; the $2 \times 2$ matrices over the semiring $\N \cup
  \{-\infty\}$ with addition $\oplus$ defined by $x\oplus y = \max(x, y)$ and
  multiplication $\otimes$ given by $x\otimes y = x + y$ (the \textit{max-plus}
  semiring); the $2\times 2$ matrices over any quotient of the max-plus semiring
  by the congruence generated by $t = t + 1$ where $t\in \N$; the $2\times
  2$ matrices over the min-plus semiring and its finite quotients by the
  congruences generated by $t = t + 1$ for all $t\in \N$; and the $n \times n$
  matrices over $\Z / n\Z$ relative to their group of units.
\end{abstract}

\tableofcontents

\section{Introduction}

In this paper we find minimum cardinality generating sets for several well-known
finite monoids of matrices over semirings. The topic of determining such
minimum cardinality generating sets for algebraic objects is classical, and has
been studied extensively in the literature; see, for
example,~\cite{Branco2019aa, Colarte2020aa, DallaVolta1995aa, Dimitrova2020aa,
  Gelander2020aa, Holt2013aa, Moori1993aa, Rosenberger1988aa}.  In this
paper we are principally concerned with monoids of matrices over the boolean
semiring $\B$; we also present some results about monoids of min-plus and
max-plus matrices of dimension $2$, and matrices of arbitrary dimension over the
rings $\Z / n\Z$, $n\in \Np$.

If $S$ is a semigroup, then the least cardinality of a generating set for $S$ is
often called the \defn{rank} of $S$ and is denoted by $\mathbf{d}(S)$. 
A \defn{semiring} is a set $\mathbb{S}$ with two operations, $\oplus$ and $\otimes$,
such that $(\mathbb{S}, \oplus)$ forms a commutative monoid with identity $e$,
$(\mathbb{S}, \otimes)$ forms a monoid, $e\otimes x = x\otimes e = e$ for all $x
\in \mathbb{S}$, and multiplication distributes over addition. We refer to $e$
and $f$ as the \defn{zero} and \defn{one} of the semiring respectively.  One
natural example of a semiring is the natural numbers $\N$; note that in this
paper $0 \in \N$ and we write $\Np$ for the set of positive natural numbers.
Another well-known example is the \defn{boolean semiring} $\B$. This is the set
$\{0, 1\}$ with addition defined by 
\begin{align*}
  0 \oplus 1 = 1 \oplus 0 = 1 \oplus 1 = 1 \AND
  0 \oplus 0 = 0
\end{align*}
and multiplication defined as usual for the real numbers $0$ and $1$. If $n\in
\Np$, then we denote by $\Bn$ the monoid consisting of all $n\times n$ matrices
with entries in $\B$.
The semiring $\B$ is one of the simplest examples of a semiring, and the matrix
monoids $\Bn$ for $n \in \N$ have been widely studied in the literature since
the 1960s to the present day;
see, for example,~\cite{Breen2001aa, Breen1997aa, Butler1974aa, Caen1981aa, Cho1993aa,
  Cho1993ab, Fenner2018aa, Kim1982aa, Konieczny1992aa, Li1995aa, Plemmons1970aa,
  Plemmons1970ab, Roush1977aa, Schwarz1973aa, Shaofang1998aa, Tan2000aa,
  Zivkovic2006aa}. If $\alpha$ is a binary relation on the set $\{1, \ldots,
  n\}$, then we can define an $n\times n$ matrix $A_{\alpha}$ with entries in
the boolean semiring $\B$ such that the entry in row $i$, column $j$ of
$A_{\alpha}$ is $1$ if and only if $(i, j) \in \alpha$. The function that maps
every binary relation $\alpha$ to the corresponding $A_{\alpha}$ is an
isomorphism  between the monoid of binary relations on $\{1,  \ldots, n\}$ and
$\Bn$. Functions are a special type of binary relations, and composition of
functions coincides with composition of binary relations when applied to
functions.  As such the monoid $\Bn$ can be thought of as a natural
generalisation of the full transformation monoid consisting of all functions
from $\{1, \ldots, n\}$ to itself, under composition of functions.  In
comparison to the full transformation monoid and its peers, such as the
symmetric inverse monoid or the so-called diagram monoids, whose structures are
straightforward to describe, $\Bn$ has a rich and complex structure.  For
example, every finite group appears as a maximal subgroup of $\Bn$ for some $n
\in \N$ (see~\cite{Clifford1970aa, Montague1969aa, Plemmons1970ab}), and the
Green's structure of $\Bn$ is highly complicated; neither the number of
$\J$-classes of $\Bn$ nor the largest length of a chain of $\J$-classes is known
for $n\geq 9$. In this landscape it is perhaps not surprising that the minimum
sizes $\mathbf{d}(\Bn)$ of generating sets for the monoids $\Bn$ were previously
unknown for $n \geq 6$. On the other hand, a description of a minimal generating
set for $\Bn$ was given by  Devadze in 1968~\cite{Devadze1968aa} (see
\cref{thm:DevadzeKon}). There is no proof in~\cite{Devadze1968aa} that
the given generating sets are minimal, a gap that was filled by Konieczny in
2011~\cite{Konieczny2011aa}. The minimal generating sets given by Devadze and
Konieczny are specified in terms of representatives of certain $\J$-classes of
$\Bn$, the exact number of which is difficult to compute, and grows very quickly
with $n$; see~\cref{cor:ExponentialGenSets}.

The monoid $\Bn$ has several natural submonoids, such as the
monoids: $\Refln$ consisting of the \defn{reflexive boolean
  matrices} (that is matrices containing the identity matrix); $\Halln$
consisting of all \defn{Hall matrices} (matrices containing a permutation); and
$\UTn$ of \defn{upper triangular boolean matrices}. Each of these submonoids has
been extensively studied in their own right; see, for example,
\cite{Butler1974aa, Cho1993ab, Gaysin2021aa, Kim1977aa, Li2011aa, Pin1985aa,
  Schwarz1973aa, Straubing1980aa, Tan2000aa, Zhang2020aa}. Many other submonoids of $\Bn$ have also been investigated, although we do not consider these submonoids in this article; for a recent example, see~\cite{Carvalho2021aa}. 
  Unlike $\Bn$,
$\Refln$ is $\J$-trivial, and so has precisely $|\Refln| = 2 ^ {n ^ 2 - n}$
$\J$-classes. Similarly, $A, B\in \Halln$ are $\J$-related if and only if one
can be obtained from the other by permuting the rows and columns. However, the
size of $\Halln$ is not known for $n\geq 8$. This question was raised by Kim Ki
Hang~\cite[Problem 13]{Kim1982aa} in 1982 and recently raised again
in~\cite{Gaysin2020aa}; see~\cite{OEISHall} for the known values. In contrast,
the upper-triangular boolean matrix monoid $\UTn$ is easily seen to have size
$2^{\frac{n(n+1)}{2}}$.
The monoid $\UTn$ appears to have been primarily studied in the context of
varieties; see for instance~\cite{Li2011aa, Zhang2020aa}. It is somewhat
surprising that there appears to be no description of the unique minimal
generating set of $\UTn$ in the literature, in particular because this minimal
generating set is more straightforward to determine than that of the other
submonoids of $\Bn$ we consider (see \cref{sec:TriBoolMat}).

We also consider monoids of matrices over certain \defn{tropical} semirings. 
The \defn{min-plus semiring} $\Kmin$ is the set $\N \cup \{\infty\}$, with
$\oplus = \min$ and $\otimes$ extending the usual addition on $\N$ so that
$x\otimes\infty = \infty\otimes x = \infty$ for all $x \in \Kmin$. The
multiplicative identity of $\Kmin$ is $0$, and the additive identity is
$\infty$.

The \defn{max-plus semiring} $\Kmax$ is the set $\N \cup \{-\infty\}$ with
$\oplus = \max$ and $\otimes$ extending the usual addition on $\N$ so that
$x\otimes-\infty = -\infty\otimes x = -\infty$ for all $x \in \Kmax$. The
one of $\Kmax$ is $0$ and the zero is $-\infty$.

These semirings give rise to two infinite families of finite quotients. The
\defn{min-plus semiring with threshold $t$}, denoted $\Kmint$, is the set $\{0,
  1, \ldots, t, \infty\}$ with operations $\oplus = \min$ and $\otimes$ defined
by
\begin{align*}
  a \otimes b = \begin{cases}
    \min(t, a + b) \quad &\text{$a \neq \infty$ and $b \neq \infty$} \\
    \infty \quad &\text{$a = \infty$ or $b = \infty$}.
  \end{cases}
\end{align*}

The \defn{max-plus semiring with threshold $t$}, denoted $\Kmaxt$, is
constructed analogously; its elements are $\{-\infty, 0, 1, \ldots, t\}$,
addition is $\max$, and multiplication is defined by $a \otimes b = \min(t, a +
b)$ for all $a, b \in \Kmaxt$.
For arbitrary $t\in \N$, the semirings $\Kmint$ and $\Kmaxt$ with threshold $t$ can
also be defined as the quotient of the corresponding infinite semirings $\Kmint$
and $\Kmax$ by the congruence generated by $(t, t + 1)$.
The max-plus and min-plus tropical semirings are often defined in the literature
to be $\mathbb{R}\cup\{-\infty\}$ or $\mathbb{R}\cup\{\infty\}$, respectively.
The monoids $M$ of matrices over these semirings are uncountably infinite, and
so every generating set for such a monoid $M$ has cardinality $|M|$. This is one
rationale for considering the semirings $\Kmax$ and $\Kmin$ in this paper.
Another reason we consider this unorthodox definition is that the results in
this paper concerning monoids of matrices over $\Kmax$ and $\Kmin$ arose
initially from computational experiments in a reimplementation of the
Froidure-Pin Algorithm~\cite{Froidure1997aa, Semigroupe}
in~\cite{Libsemigroups}, and the original implementation by Jean-Eric Pin
in~\cite{Semigroupe} included support for the monoids we consider here.  Note
that under the standard definition, the min-plus and max-plus semirings are
isomorphic under the map $x \to -x$; this is not the case under our definition.
There is a significant amount of literature on matrices over the tropical
semirings; see for example \cite{Daviaud2018aa, Hollings2012aa, Izhakian2010aa,
  Johnson2013aa, Johnson2014aa, Johnson2019aa, Shitov2014aa, Simon1994aa}.
However, there is relatively little literature on the classical
semigroup-theoretic properties of the tropical matrix monoids like Green's
relations or generating sets; indeed descriptions of the Green's relations were
only published relatively recently (see~\cite{Hollings2012aa, Johnson2013aa}).
This may simply be because the monoids are rather complex and difficult to work
with.  

The final semiring that we are concerned with is $\Z / n\Z$, the ring of
integers modulo $n$. For brevity we will use $\Z_n$ to denote this ring.
The monoid of matrices over $\Z_n$ has received relatively little attention as a
semigroup. As an example of an \defn{exact semiring} as defined in
\cite{Wilding2013aa}, there is a characterisation of Green's relations on
$M_k(\Z_n)$ in terms of row and column spaces, but there appears to be little
more known about the structure of $M_k(\Z_n)$.

In \cref{sec:Preliminaries} we provide the well-known but necessary
preliminaries on Green's relations and matrix monoids; this is followed in
\cref{sec:MinimalGenSets} by some elementary lemmas on generating sets
and their minimality.

In \cref{sec:boolmat} we describe and compute minimal generating sets for
$\Bn$, $\Refln$, $\Halln$, and $\UTn$; the section is organised as follows. In
\cref{sec:BMatPrelim} we introduce the necessary background, which will
be used throughout \cref{sec:boolmat}. In \cref{sec:FullBoolMat},
we describe two methods for computing minimal generating sets of the form
described by Devadze and Konieczny, and apply these methods to compute
$\d(M_6(\B))$, $\d(M_7(\B))$, and $\d(M_8(\B))$. In
\cref{sec:RefBoolMat}, we describe the unique minimal generating set for
$\Refln$ and a method for computing it, and apply this method to compute
$\d(\Refln)$ for $n \leq 7$. In \cref{sec:HallBoolMat}, we describe how
minimal generating sets for $\Halln$ arise from minimal generating sets for
$\Bn$. Finally, in \cref{sec:TriBoolMat} we describe the unique minimal
generating sets for $\UTn$ and show that $\d(\UTn) = n(n+1)/2 + 1$, one greater
than the $n$th triangular number, for all $n \in \Np$.

In \cref{sec:Tropical} we show that the generating sets for the dimension
$2$ tropical matrix monoids given in~\cite{East2020aa} are minimal, and so, in
particular, $\d(M_2(\Kmint)) = t + 4$ and $\d(M_2(\Kmaxt)) = (t^2 + 3t + 8)/2$.

Finally, in \cref{sec:IntegersModN} we describe generating sets for
$M_k(\Z_n)$ relative to the group of units $GL_k(Z_n)$, and prove that these
generating sets are minimal.

\section{Preliminaries}
\label{sec:Preliminaries}

\subsection{Green's relations}
On any semigroup $S$, there are some key equivalence relations. These are known
as Green's relations, and are defined in terms of principal ideals as follows.
Let $S$ be any semigroup and let $s, t \in S$. We denote by $S^1$ the monoid
obtained by adjoining an identity to $S$. Then
\begin{align*}
  s \L t &\text{ if and only if } S^1 s = S^1 t \\
  s \R t &\text{ if and only if } s S^1 = t S^1 \\
  s \J t &\text{ if and only if } S^1 s S^1 = S^1 t S^1.
\end{align*}
Finally, we define Green's $\H$-relation as the intersection of $\L$ and $\R$.
We write $X_s$ for the Green's $\mathcal{X}$-class of $s$, and $S/\mathcal{X}$
for the set of Green's $\mathcal{X}$-classes of $S$, where $\mathcal{X} \in
\{\L, \R, \J, \H\}$.
There is a natural partial order on certain Green's classes given by
\begin{align*}
  R_s \leq R_t &\text{ if and only if } sS^1 \subseteq tS^1 \\
  L_s \leq L_t &\text{ if and only if } S^1s \subseteq S^1t \\
  J_s \leq J_t &\text{ if and only if } S^1 s S^1 \subseteq S^1 t S^1.
\end{align*}
Further background on Green's relations can be found in~\cite{Howie1995aa}.

\subsection{Matrix semigroups}
\label{sec:matsemigp}
Given a semiring $\mathbb{S}$ and $m, n \in \Np$, we may study the set of $m
\times n$ matrices over $\mathbb{S}$; we will denote this by $M_{m,
  n}(\mathbb{S})$. In particular, we are interested in the multiplicative monoid
of square matrices of dimension $n \in \Np$ over $\mathbb{S}$, denoted by
$M_n(\mathbb{S})$.  There are a number of common features of such matrix
monoids. As usual, we let $0$ denote the additive identity of $\mathbb{S}$ and
$1$ denote the multiplicative identity. The following definitions are entirely
analogous to those from the standard linear algebra of vector spaces over
fields.

Let $A \in M_n(\mathbb{S})$. We write $A_{i*}$ to denote the $i$th row of $A$,
$A_{*i}$ to denote the $i$th column, and $A_{ij}$ to denote the $j$th entry of
the $i$th row. A \defn{linear combination} of rows is a sum of scalar multiples
of the rows, where both operations are defined componentwise in the usual way.
The \defn{row space} $\RowS(A)$ of $A$ is the set of all linear combinations of
the rows of $A$. An \defn{isomorphism} of row spaces is a linear bijection; that
is, a bijection $T: \RowS(A) \to \RowS(B)$ such that $T(\alpha x + \beta y) =
\alpha T(x) + \beta T(y)$ for all $x, y \in \RowS(A)$ and for all $\alpha, \beta
\in \mathbb{S}$.  A set of rows is \defn{linearly independent} if no row can be
written as a linear combination of other rows; this coincides with the usual
definition when $\mathbb{S}$ is a field. A \defn{spanning set} for a row space
is a set of rows which every element of the row space may be written as a linear
combination of.  A \defn{row basis} of $A$ is then a linearly independent
spanning set for the row space of $A$. \defn{Column spaces} $\ColS(A)$ and
\defn{column bases} $\ColB(A)$ are defined dually.\footnote{This notation for
  row and column spaces and bases arises from their application in acting
  semigroup algorithms; see~\cite{East2019aa} for some details although not in
  the particular context of matrix semigroups.}
The importance of row and column bases arises from the following well-known
results:

\begin{prop}
  \label{prop:rowsandideals}
  Let $A, B \in M_n(\mathbb{S})$. Then the following hold:
  \begin{enumerate}[label={\rm (\roman*)}]
    \item 
      $\RowS(AB) \subseteq \RowS(B)$ and $\ColS(AB) \subseteq \ColS(A)$.
    \item
      $L_A \leq L_B$ if and only if $\RowS(A) \subseteq \RowS(B)$ and $R_A \leq
      R_B$ if and only if $\ColS(A) \subseteq \ColS(B)$.
  \end{enumerate}
\end{prop}
\begin{proof}
  We will only prove the statements related to row spaces and $\L$-classes; the
  other statements may be proved in a dual fashion.
  \noindent \textbf{(i).}
  Suppose that $A = [\alpha_{ij}]$ and $B = [\beta_{ij}]$. If 
  $AB = [\gamma_{ij}]$, then 
  \[
  \gamma_{ij} = \alpha_{i1} \beta_{1j} + \alpha_{i2}\beta_{2j} 
                + \cdots + \alpha_{in}\beta_{nj}
              = \sum_{k = 1} ^ {n} \alpha_{ik}\beta_{kj}.
  \]
  It follows that the $i$th row of $AB$ is 
  \[
  (\gamma_{i1}, \gamma_{i2}, \ldots, \gamma_{in})
   =  \left (\sum_{k = 1} ^ {n} \alpha_{ik}\beta_{k1},
     \sum_{k = 1} ^ {n} \alpha_{ik}\beta_{k2},
     \ldots, 
     \sum_{k = 1} ^ {n} \alpha_{ik}\beta_{kn}\right) \\
    =
    \sum_{k = 1} ^ {n} 
    \alpha_{ik}B_{k*}.
  \]
  Thus the $i$th row of $AB$ is a sum of rows of $B$, and so $\RowS(AB)
  \subseteq \RowS(B)$.

  \bigskip
  
  \noindent \textbf{(ii).}
  If $L_A \leq L_B$, then there exists $X \in M_n(\mathbb{S})$ such that $A =
  XB$; hence $\RowS(A) = \RowS(XB) \subseteq \RowS(B)$ by (i).  Conversely, if
  $\RowS(A) \subseteq \RowS(B)$, then every row of $A$ can be expressed as a
  linear combination of rows of $B$. For $1 \leq i \leq n$ of $A$, let $A_{i*} =
  \sum_{j = 1}^{n}x_{ij}B_{j*}$, and define $X = [x_{ij}]$. Then $A = XB$, and
  hence $L_A \leq L_B$. 
\end{proof}

The following corollary now follows immediately from \cref{prop:rowsandideals}.

\begin{cor} 
  \label{lem:GreensRowColumnSpaces}
  Let $A, B \in M_n(\mathbb{S})$. Then $A \L B$ if and only if $\RowS(A) =
  \RowS(B)$, and $A \R B$ if and only if $\ColS(A) = \ColS(B)$. 
\end{cor}

The following lemma provides a useful connection between the $\J$-relation and row
spaces; see \cref{thm:Zaretskii} for a strengthening in the case of $\Bn$.
The proof of this lemma is essentially the same as for boolean matrices,
as found in~\cite[Theorem 1.3.3, forward direction]{Kim1982aa}. Note that this
theorem and proof is in terms of Green's $\D$-relation (see \cite{Howie1995aa}),
but for the finite monoids $M_{n}(\mathbb{S})$, $\D = \J$.
\begin{prop}
  \label{prop:RowSpacesJRelation}
  Let $\mathbb{S}$ be a finite semiring. If two matrices in
  $M_{n}(\mathbb{S})$ are $\J$-related, then they have isomorphic row spaces.
\end{prop}

The \defn{group of units} of $M_n(\mathbb{S})$ is the group of invertible
matrices (\defn{units}) in $M_n(\mathbb{S})$; the identity matrix is always a
unit. The group of units is always the maximal class in the $\J$-order of
$M_n(\mathbb{S})$.

For any semiring $\mathbb{S}$, the symmetric group embeds into the group of units of
$M_n(\mathbb{S})$ by the map 
\begin{align*}
  \phi:\: &\alpha \to [a_{ij}], \\
  &a_{ij} =
    \begin{cases}
      1 \quad & i\alpha = j \\ 
      0 \quad &\text{otherwise}.
    \end{cases}
\end{align*}
The image of this embedding will often simply be referred to as $S_n$ or the
symmetric group when context prevents ambiguity. Elements of this embedding
$S_n$ are called \defn{permutation matrices}; multiplying by a permutation
matrix on the left permutes rows of a matrix, and multiplying on the right
permutes columns. Similarly, we may define a \defn{transformation matrix} to be
one which contains a single $1$ in every row; these are the images of the
obvious extension of $\phi$ to the full transformation monoid $\mathcal{T}_n$.

Two matrices $A, B\in M_n(\mathbb{S})$ are \defn{similar} if each can be
obtained from the other by permuting the rows and/or columns. Note that similar
matrices are $\J$-related in $M_n(\mathbb{S})$, by multiplying on the left
and/or right by permutation matrices.

A non-unit matrix $A \in M_n(\mathbb{S})$ is \defn{prime} if $A = BC$ implies
$B$ or $C$ is a unit. The prime matrices of $M_n(\mathbb{S})$ are immediately
below the group of units in the $\J$-order. 

\subsection{Minimal generating sets}
\label{sec:MinimalGenSets}

Given a semigroup $S$, we may ask what the minimum cardinality is of a
generating set for $S$. This is known as the \defn{rank} of $S$ and is denoted
by $\mathbf{d}(S)$. The same notion exists for monoid generating sets; the rank
of a monoid as a monoid is either one less than the rank as a semigroup (if
the group of units is trivial), or the same otherwise. We therefore consider
only the rank of monoids as semigroups in this paper. A generating
set $X$ for $S$ is \defn{irredundant} if no subset of $X$ generates $S$, and we
say that any irredundant generating set of cardinality $\mathbf{d}(S)$ is a
\defn{minimal generating set}. If $\mathbf{d}(S) \in \N$, then every generating
set of cardinality $\mathbf{d}(S)$ is irredundant and hence minimal; this does
not hold when $\mathbf{d}(S)$ is infinite. We will also say that a generator $x
\in X$ is irredundant if $X \setminus \{x\}$ does not generate $S$.

An element $x$ of a monoid $M$ with identity $e$ is \defn{decomposable} if $x$
may be written as a product of elements in $M\setminus\{e, x\}$, and
\defn{indecomposable} otherwise. Note that a different definition of
indecomposable elements exists in the literature. Under that definition, an
element $x$ of a semigroup $S$ is decomposable if $x \in S^2 = \set{st}{s, t \in
  S}$, and indecomposable otherwise. These definitions are closely related but
not equivalent, and our definition will be more useful for the purposes of this
paper.

Note that every indecomposable element of a monoid $M$ must be contained in
every generating set for $M$; hence the minimal generating set of a monoid is a
superset of the indecomposable elements.

We will use the following straightforward lemma repeatedly to show that certain
sets generate submonoids of $\Bn$.
\begin{lemma}
  \label{lem:GenSetDecomposition}
  Let $S$ be a finite semigroup and let $X$ be a subset of $S$, such that for
  every element $x \in X$ with $\J$-class $J_x$, we have $J_x \subseteq
  \genset{X}$. If every element $s \in S$ that is not $\J$-related to an element
  of $X$ can be written as a product of elements of $S$, none of which are
  $\J$-related to $s$ in $S$, then $X$ generates $S$. 
\end{lemma}
\begin{proof}
  Let $s \in S$. Then either $s$ is in a $\J$-class of an element belonging to
  $X$, in which case we are done, or $s$ may be written as a product of elements
  which are not $\J$-related to $s$, and hence lie strictly above $s$ in the
  $\J$-order. This same argument applies to each of those elements. Since $S$
  has finitely many $\J$-classes, this process must eventually terminate.
  However, the process of decomposing elements may only terminate if each
  element to be decomposed lies in the $\J$-class of an element of $X$; hence $s
  \in \genset{X}$.
\end{proof}

We also have the following proposition, used to prove that certain generating sets are
minimal. 
\begin{prop}
  \label{prop:Wilf}
  Let $S$ be a finite semigroup. Suppose that $X$ is an irredundant generating
  subset of $S$ that contains at most one element from each $\J$-class of $S$.
  Then $X$ has minimum cardinality, i.e.\ $\mathbf{d}(S) = |X|$.
\end{prop}
\begin{proof}
  Let $x \in X$.
  It suffices to show that any generating set for $S$ contains an
  element of $J_{x}$.

  Let $U$ be the union of the $\J$-classes of $S$ that lie strictly above
  $J_{x}$ in the $\J$-order.  If $x$ is written as a product of elements of $S$,
  then it is clear that those elements are contained in $J_{x} \cup U$.
  Furthermore, since $X$ generates $S$, which contains $U$, and since no element
  of $U$ can be written as a product involving $x$, it follows that $U \subseteq
  \genset{X \setminus \{x\}}$.  Therefore $\genset{U} \leq \genset{X \setminus
  \{x\}}$. The assumption that $X$ is irredundant implies that $x \not\in
  \genset{X \setminus \{x\}}$, and so $x \not\in \genset{U}$.  Therefore, if $x$
  is written as a product of elements of $S$, then at least one of those
  elements is contained in $J_{x}$.  In other words, any generating set for $S$
  contains an element of $J_{x}$.
\end{proof}


\section{Boolean matrix monoids}
\label{sec:boolmat}

In this section, we compute minimal generating sets for the full boolean matrix
monoid $\Bn$ and some of its natural submonoids. In order to define these
submonoids, we must introduce the concept of matrix containment: given two
matrices $A, B \in \Bn$, we say that $A$ is \emph{contained} in $B$ if for all
$1 \leq i,j \leq n$, $A_{ij} = 1$ implies that $B_{ij} = 1$; that is, $B$ contains
a $1$ in every position in which $A$ does.

We consider the submonoids:
\begin{enumerate}[label={\rm (\roman*)}]
  \item $\Refln$, of matrices containing the identity (reflexive matrices),
  \item $\Halln$, of matrices containing a permutation matrix (Hall matrices),
  \item $\UTn$ and $\LTn$, of upper- and lower-triangular boolean matrices.
\end{enumerate}

As mentioned above, there is a natural isomorphism between binary relations on
$\{1, \ldots, n\}$ and $\Bn$.  The monoid $\Refln$ arises as the image of the
reflexive binary relations under this isomorphism, and this is the primary
reason for studying $\Refln$. However, it may also be viewed as the monoid of
matrices containing a certain pattern of $1$s - namely the identity matrix.
Similarly, $\Halln$ arises as those matrices describing a family of subsets
which satisfy Hall's marriage condition, but also as those matrices containing a
permutation matrix.

We compute the largest known ranks of these monoids, which are presented in
\cref{tab:BMatResults} along with whether that rank was previously known or
is practically computable by brute force. The rank of $\UTn$ is given by the
triangular numbers plus $1$; this may have already been known but we
were unable to find a reference in the literature. Note that since $\UTn$ and
$\LTn$ are anti-isomorphic via the transposition map, $\mathbf{d}(\UTn) =
\mathbf{d}(\LTn)$ for all $n$. None of the other ranks are known beyond $n = 8$. 

\begin{table}
  \centering
  \begin{tabular}{l|r|r|r|r|r}
    $n$ & $\mathbf{d}(\Bn)$ & $\mathbf{d}(\Refln)$ & $\mathbf{d}(\Halln)$ &
    $\mathbf{d}(\MTn)$ & $\mathbf{d}(\UTn)$ \\ 
    \hline
    1 & *2         & *1& *1& *1& *3\\
    2 & *3         & *2& *2& *3& *4\\
    3 & *5         & *9& *4& *5& *7\\
    4 & *7         & *39& *6& *7& *11\\
    5 & *13        & *1\ 415& *12& 13& *16\\
    6 & 68         & 482\ 430& 67& ?& *22\\
    7 & 2\ 142     & 1\ 034\ 972\ 230& 2\ 141& ?& 29\\
    8 & 459\ 153   & ?& 459 152& ?& 37\\
    9 & ?   & ?& ?& ? & 45
  \end{tabular}
  \vspace{1cm}

  \caption{The ranks of certain matrix monoids over $\B$; a * denotes that the
    rank was already known or is computable by brute force, and a ?\ indicates
    that the value is unknown.
    For $\MTn$, see the proposition below and the discussion preceding it.
    Note that the ranks given are to generate the monoids
    as semigroups. The values for $\mathbf{d}(\Bn)$ and $\mathbf{d}(\Refln)$ can be found in the OEIS:~\cite{OEISfullbool, OEISreflex}.}
  \label{tab:BMatResults}
\end{table}
There are several obvious relationships between columns in
\cref{tab:BMatResults}; we prove that these relationships hold for all $n
\in \Np$.

Another submonoid of $\Bn$, which has attracted recent attention, is the Gossip
monoid $G_n$; see~\cite{Brouwer2015aa, Fenner2018aa} for some recent results and
\cite{Baker1972aa, Hajnal1972aa} for the problem which inspired $G_n$. This is
typically defined as the monoid generated by the so-called \defn{phone-call}
matrices
$C[i, j]$, where 
\[C[i, j]_{kl} = 
    \begin{cases}
      1 \quad &\text{ if } k = l \text{ or } \{i, j\} = \{k, l\} \\ 
      0 \quad &\text{otherwise}
    \end{cases}
\]
and where $1 \leq i, j \leq n$.
It is straightforward to show that the phone-call matrices form a minimal
generating set for $G_n$: since every element of $G_n$ contains the identity
matrix and at least two more $1$s, the phone-call matrices are indecomposable,
and hence contained in any generating set.

There are a number of other submonoids of $\Bn$ that arise as generalisations of
the Hall monoid. If we define the \defn{containment closure} $\bar{S}$ of a
subsemigroup $S \leq \Bn$ to be the set of matrices $A \in \Bn$ containing some
element of $S$, then for any choice of subsemigroup $S$, $S \leq \bar{S} \leq
\Bn$. In particular, the submonoids $\Refln$ and $\Halln$ are of this
type. It remains an open problem to determine minimal generating sets for
arbitrary containment closures. This problem seems difficult; a more tractable
problem may be to restrict $S$ to subgroups of $S_n$ or submonoids of
transformation matrices. In particular, if we choose $S$ to be the set of
transformation matrices, we obtain the monoid $\MTn$ of matrices containing a
transformation matrix. The following is an unpublished result of Marcel Jackson.
\begin{prop*}[\cite{Jackson2021aa}]
  \label{conj:TransGenSet}
  For $n \geq 2$, every minimal generating set for $\MTn$ consists of a set of
  representatives of the $\J$-classes of $\Bn$ which contain a prime matrix,
  together with a minimal generating set for $S_n$, and two matrices similar to
  \begin{align*}
  \begin{pmatrix}
    0 & 1 & 0 & 0 & \cdots & 0 \\
    0 & 1 & 0 & 0 & \cdots & 0 \\
    0 & 0 & 1 & 0 & \cdots & 0 \\
    0 & 0 & 0 & 1 & \cdots & 0 \\
    \vdots  & \vdots & \vdots & \vdots & \ddots & \vdots\\
    0 & 0 & 0 & 0 & \cdots & 1
  \end{pmatrix}\text{ and } 
  \begin{pmatrix}
    1 & 1 & 0 & 0 & \cdots & 0 \\
    0 & 1 & 0 & 0 & \cdots & 0 \\
    0 & 0 & 1 & 0 & \cdots & 0 \\
    0 & 0 & 0 & 1 & \cdots & 0 \\
    \vdots  & \vdots & \vdots & \vdots & \ddots & \vdots\\
    0 & 0 & 0 & 0 & \cdots & 1
  \end{pmatrix}.
  \end{align*}
\end{prop*}

For $n > 2$, a minimal generating set for $S_n$ has size $2$ (for example, elements
corresponding to a tranposition and an $n$-cycle); for $n \leq 2$ one element is
sufficient to generate $S_n$.

\subsection{Preliminaries}
\label{sec:BMatPrelim}

It is known that $\d(\Bn)$ grows super-exponentially with $n$; see
\cref{cor:ExponentialGenSets}. In contrast, the subsemigroup of $\Bn$ generated
by all the regular elements can be generated by the four matrices

\begin{align}
  \label{eq:RegularGens}
  \begin{split}
  T = \begin{pmatrix}
    0 & 1 & 0 & 0 & \cdots & 0 \\
    1 & 0 & 0 & 0 & \cdots & 0 \\
    0 & 0 & 1 & 0 & \cdots & 0 \\
    0 & 0 & 0 & 1 & \cdots & 0 \\
    \vdots  & \vdots & \vdots & \vdots & \ddots & \vdots\\
    0 & 0 & 0 & 0 & \cdots & 1 
  \end{pmatrix}\text{, }&
  U = \begin{pmatrix}
    0 & 1 & 0 & 0 & \cdots & 0 \\
    0 & 0 & 1 & 0 & \cdots & 0 \\
    0 & 0 & 0 & 1 & \cdots & 0 \\
    0 & 0 & 0 & 0 & \cdots & 0 \\
    \vdots  & \vdots & \vdots & \vdots & \ddots & \vdots\\
    1 & 0 & 0 & 0 & \cdots & 0 
  \end{pmatrix},\\
  E = \begin{pmatrix}
    1 & 0 & 0 & 0 & \cdots & 0 \\
    1 & 1 & 0 & 0 & \cdots & 0 \\
    0 & 0 & 1 & 0 & \cdots & 0 \\
    0 & 0 & 0 & 1 & \cdots & 0 \\
    \vdots  & \vdots & \vdots & \vdots & \ddots & \vdots\\
    0 & 0 & 0 & 0 & \cdots & 1 
  \end{pmatrix}\text{, }&
  F = \begin{pmatrix}
    0 & 0 & 0 & 0 & \cdots & 0 \\
    0 & 1 & 0 & 0 & \cdots & 0 \\
    0 & 0 & 1 & 0 & \cdots & 0 \\
    0 & 0 & 0 & 1 & \cdots & 0 \\
    \vdots  & \vdots & \vdots & \vdots & \ddots & \vdots\\
    0 & 0 & 0 & 0 & \cdots & 1
  \end{pmatrix};
\end{split}
\end{align}
see~\cite{Roush1977aa} for further details.
We will continue to use these names for these matrices throughout
\cref{sec:boolmat}. Note that together $T$ and $U$ minimally generate
$S_n$ for $n \geq 2$; they correspond to the transposition $(1~2)$ and the
$n$-cycle $(1~2~\ldots~n)$ respectively. All results which involve $T$ and $U$
in this paper continue to hold if $T$ and $U$ are replaced with any two elements
which together generate $S_n$.

We call any matrix similar to $E$ an \defn{elementary} matrix. In particular,
the elementary matrix which consists of the identity matrix with an additional
$1$ in position $j$ of the $i$th row will be denoted by $E^{i,j}$. If $J_E$ is
the $\J$-class of $E$ in $\Bn$, then it is easy to show that $J_E =
\set{E^{i,j}}{1 \leq i, j \leq n}$ and we call $J_E$ the \defn{elementary
  $\J$-class}. 
\begin{thm}[Devadze's Theorem~\cite{Konieczny2011aa}]
  \label{thm:DevadzeKon}
  For $n > 2$, the set $\{T, U, E, F\} \cup P$ is a generating set for $\Bn$ of
  minimum cardinality, where $P$ is any set containing one matrix from every
  $\J$-class of $\Bn$ which contains a prime matrix. 
\end{thm}
By Devadze's Theorem, in order to determine the rank of $\Bn$ it is sufficient
to compute a set of representatives of the $\J$-classes of $\Bn$ containing a
prime matrix; see \cref{sec:FullBoolMat} for details of this computation.

As well as the description of a minimal generating set for $\Bn$, there are
certain additional concepts attached to $\Bn$ which do not apply to matrices
over arbitrary semirings. Many of these arise from the observation that we may
view a vector $v \in \B^n$ as the characteristic function of a subset $s(v)$ of
$\{1, \ldots, n\}$, where $i \in s(v)$ if and only if $v_i = 1$.  Then we define
the \defn{union} of two vectors $v, w$ to be $s^{-1}(s(v) \cup s(w))$. Note that
the union and sum of two vectors coincide in $\B^n$. In fact, since there is a
single non-zero element of $\B$, unions, sums, and linear combinations all
coincide. Given $v, w \in \B^n$, we also say that $v$ is \defn{contained} in
$w$, and write that $v \leq w$, if $s(v) \subseteq s(w)$.

It is straightforward to verify that if $A \in \Bn$, then there is a unique row
basis $\RowB(A)$, which is the unique minimal generating set for $\RowS(A)$
under union, consisting of the non-zero $\leq$-minimal rows. The dual statement
for column spaces also holds. This yields the following more practical
counterpart to \cref{lem:GreensRowColumnSpaces}, as an immediate consequence of
that corollary.

\begin{prop} 
  \label{lem:BMatGreensRowColumnBases}
  Let $A, B \in \Bn$. Then $A \L B$ in $\Bn$ if and only if $\RowB(A) =
  \RowB(B)$, and $A \R B$ in $\Bn$ if and only if $\ColB(A) = \ColB(B)$.
\end{prop}
There is a simple algorithm to compute the row and column bases of matrices in
$\Bn$ in time and space cubic in $n$, and hence the previous proposition gives
an efficient method for determining whether two matrices are $\L$- or
$\R$-related in $\Bn$. This, in combination with
\cref{lem:GreensRowColumnSpaces}, allows for the efficient computation of the
$\L$- and $\R$-structure of $\Bn$. However, the number of $\L$- and $\R$-classes
grows extremely rapidly with $n$, as shown in \cref{tab:NumberLRClasses},
so that it quickly becomes infeasible to compute the $\L$- and $\R$-structure of
$\Bn$. Note that transposition gives an anti-automorphism from $\Bn$ to itself
which exchanges $\L$- and $\R$-classes, and hence determining the $\L$- and
$\R$-structure only requires computation of one of the relations.

\begin{table}
  \centering
  \begin{tabular}{r|r|r|r|r|r|r|r}
    $n$ & 1 & 2 & 3 & 4 & 5 & 6 & 7 \\
    \hline
    $|\Bn / \L|$ & 2 & 7 & 55 & 1\ 324& 120\ 633& 42\ 299\ 663& ?
  \end{tabular}
\vspace{1cm}

\caption{The number of $\L$- or $\R$-classes in $\Bn$.} 
  \label{tab:NumberLRClasses}
\end{table}

Even determining the possible cardinalities of row spaces of matrices in $\Bn$
is a hard problem: see~\cite{Breen2001aa, Konieczny1992aa, Li1995aa,
  Shaofang1998aa, Zivkovic2006aa}. This is in stark contrast to the row-spaces
of matrices over fields, which have dimension a power of the cardinality of the
field. For example, the matrix \[\mat{0}{1}{1}{1} \in \Bn\] has row space
cardinality $3$.

It is significantly more difficult to determine the $\J$-relation in $\Bn$ than
the $\L$- or $\R$-relation. Indeed, the problem of determining whether two
matrices are $\J$-related in $\Bn$ is NP-hard; see~\cite[Theorem 2.7]{Fenner2018ab}
and ~\cite{Markowsky1992aa}.
The $\J$-relation on $\Bn$ is characterised by row space embeddings in the
following theorem; a function $f: \RowS(A) \to \RowS(B)$ is a \defn{row space
  embedding} if it respects containment and non-containment, i.e.\ $f(v) \leq
f(w)$ if and only if $v \leq w$ for all $v, w \in \RowS(A)$.
\begin{thm}[Zaretskii's Theorem~\cite{Zaretskii1963aa}]
  \label{thm:Zaretskii}
  Let $A, B \in \Bn$. Then $J_A \leq J_B$ if and only if there exists a row
  space embedding $f: \RowS(A) \to \RowS(B)$.
\end{thm}
Zaretskii's Theorem reduces the problem of determining the $\J$-order of $\Bn$
to the problem of digraph embedding, but it should be noted that the size of
$\RowS(A)$ is bounded above by $2^n$, and equality is possible. Computing the
$\J$-structure of $\Bn$ is challenging; see~\cite{Breen1997aa}. The largest $n$
for which the number of $\J$-classes of $\Bn$ is known is $8$;
see~\cite{Breen2001aa}.

Considering rows as subsets of $\{1, \ldots, n\}$, it is natural to consider
when rows of a matrix $A \in \Bn$ are contained in other rows of $A$. We will
call $A$ \defn{row-trim} if no non-zero row of $A$ is contained in another row.
\defn{Column-trim} is defined dually. We say that $A$ is \defn{trim} if it is
both row-trim and column-trim.

Our interest in trim matrices is due to the following result.

\begin{lemma}[{\cite[Lemma 3.1]{Konieczny2011aa}}]
  \label{lem:PrimeMatricesAreTrim}
  Every prime matrix in $\Bn$ is trim.
\end{lemma}
\begin{proof}
  Let $A \in \Bn$ be prime. Suppose some non-zero row indexed $k$ of $A$ is
  contained in a row indexed $l$ of $A$. Define $X \in \Bn$ to be the matrix
  such that $X_{ij} = 1$ precisely if $A_{j*} \leq A_{i*}$. Then $XA = A$, and
  since $|X_{l*}| \geq 2$, $X \not\in S_n$, a contradiction. A dual argument
  shows that no non-zero column of $A$ is contained in another column.
\end{proof}

It is difficult to enumerate prime matrices directly, but comparatively simple
to enumerate trim matrices. This is key to our strategy for computing a minimal
generating set for $\Bn$.

The technique used to define the matrix $X$ in the proof of
\cref{lem:PrimeMatricesAreTrim} will be useful throughout. Given two matrices
$A, B \in \Bn$, we say that the \defn{greedy left multiplier} of $(A, B)$ is the
matrix $C$ containing a $1$ in position $j$ of row $i$ if and only if $A_{j*}
\leq B_{i*}$. The \defn{greedy right multiplier} of $(A, B)$ is the matrix $D$
containing a $1$ in position $j$ of row $i$ if and only if $A_{*i} \leq B_{*j}$.
Observe that if $\RowB(A)$ is contained in $\RowS(B)$, then every row $v$ of $A$
may be written as the linear combination of those rows of $B$ which are
contained in $v$; hence $A = CB$ where $C$ is the greedy left multiplier of $(A,
B)$. Combining this observation with \cref{lem:GreensRowColumnSpaces} yields
the following lemma:
\begin{lemma}
  \label{lem:GreedyMultipliers}
  For any $A, B \in \Bn$ and $C$ the greedy left multiplier of $(A, B)$, the
  following are equivalent:
  \begin{enumerate}[label={\rm (\roman*)}]
    \item $A = CB$
    \item $L_A \leq L_B$
    \item $\RowS(A) \subseteq \RowS(B)$.
  \end{enumerate}
  The dual statement holds for greedy right multipliers, Greens $\R$-order, and
  column spaces.
\end{lemma}

A similar property to being trim is being \defn{reduced}. A matrix $A \in \Bmn$
is \defn{row-reduced} if no row of $A$ can be written as a union of other rows
of $A$. \defn{Column-reduced} is defined dually. We say that $A$ is
\defn{reduced} if it is both row-reduced and column-reduced. Since no row of a
trim matrix is contained in another row, it follows that no row can be expressed
as a union of other rows. A dual argument applies to columns, and hence we have
the following lemma.

\begin{lemma}
  \label{lem:TrimMatricesAreReduced}
  Every trim matrix is reduced.
\end{lemma}

The following lemma describes the $\J$-relation on reduced matrices in $\Bn$.

\begin{lemma}[{\cite[Theorem 1.8]{Plemmons1970aa}}]
  \label{lem:PermutingReducedMatrices}
  Let $A, B \in \Bn$ be reduced. Then $A \J B$ if and only if $A$ and $B$ are
  similar.
\end{lemma}

Due to the previous lemma, reduced matrices are particularly convenient to
compute with. A linear reduction to graph isomorphism, described in
\cref{sec:FullBoolMat}, shows that the problem of determining whether two
reduced matrices are $\J$-related has at most the same complexity as graph
isomorphism.  A recent paper of Babai claims that this complexity is at most
quasi-polynomial; see~\cite{Babai2016aa}.
The previous lemma also implies that the $\J$-class of any prime matrix $P$
consists of prime matrices similar to $P$, and we refer to such a $\J$-class as
a \defn{prime $\J$-class}. Note that a prime $\J$-class therefore contains at
most $(n!)^2$ elements. This observation, combined with Devadze's Theorem, has
the following corollary, which was previously known but does not appear to have
been published.

\begin{cor}
  \label{cor:ExponentialGenSets}
  The size $\mathbf{d}(\Bn)$ of a minimal generating set for $\Bn$ grows
  super-exponentially with $n$.
\end{cor}
\begin{proof}
  This follows from the fact that there at least $2^{\frac{n^2}{4} - O(n)}$
  prime boolean matrices in $\Bn$ (see~\cite[Theorem 2.4.1]{Kim1982aa}) and each
  prime $\J$-class contains at most $(n!)^2$ elements; hence there are
  super-exponentially many prime $\J$-classes.
\end{proof}

The prime matrices of $M_n(\mathbb{S})$ sit directly below the group of units in
the $\J$-order on $M_n(\mathbb{S})$ for any semiring $\mathbb{S}$. In the case
of $\Bn$, there is an additional $\J$-class immediately below $S_n$: the
$\J$-class $J_E$ of the elementary matrix $E$ from \eqref{eq:RegularGens}. The
next result shows that in fact these are all of the $\J$-classes immediately
below $S_n$.

Let $\beta_n$ denote the set $\set{\RowS(A)}{A\in \Bn\setminus S_n}$ of all
possible proper row subspaces of elements of $\Bn$, and let $\B^n$ denote
the space of all boolean vectors of length $n$. Note that the matrices in $\Bn$ with row
space equal to $\B^n$ are the permutation matrices.

\begin{thm}[cf.\ Theorem 5.1 in~\cite{Caen1981aa}]
  \label{thm:MaximalRowSpaces}
  Let $A \in \Bn\setminus S_n$. Then $\RowS(A)$ is maximal with respect to
  containment in $\beta_n$ if and only if $A$ is prime or elementary.  
\end{thm}

We may now prove a slightly stronger form of Devadze's Theorem.

\begin{thm}[Devadze's Theorem~\cite{Konieczny2011aa}]
  \label{thm:Devadzefull}
  For $n > 2$, any minimal generating set for $\Bn$ is given by $\{T', U', E',
    F'\} \cup P$, where $T'$ and $U'$ generate $S_n$, $E'$ is elementary, $F'$
  is a matrix similar to $F$, and $P$ is a set of representatives of the prime
  $\J$-classes of $\Bn$. Conversely, any such set generates $\Bn$ minimally.
\end{thm}
\begin{proof}
  The converse follows immediately from \cref{thm:DevadzeKon}, noting that $A
  \in \genset{A', T', U'}$ for any similar matrices $A, A' \in \Bn$.

  Let $X$ be a minimal generating set for $\Bn$. Since $\mathbf{d}(S_n) = 2$,
  and $X$ must contain generators of the group of units $S_n$, it follows that
  $X$ contains two elements which together generate $S_n$, which we denote by
  $T'$ and $U'$. By~\cite[Lemma 4.2]{Konieczny2011aa}, $X$ also contains a set
  $P$ of representatives of the prime $\J$-classes of $\Bn$. By~\cite[Lemma
  4.5]{Konieczny2011aa}, and the fact that the elementary $\J$-class lies
  immediately below the group of units, $X$ must also contain an elementary
  matrix, say $E'$. It only remains to show that $X$ must contain a matrix
  similar to $F$. Since none of the elements of $P$, nor $E'$, contain a
  zero row, $P \cup \{E'\}$ does not generate $F$, and $X$ must contain a matrix
  with at least one zero row. Since matrices similar to $F$ have the maximal row
  spaces amongst matrices containing zero rows, $X$ must contain a matrix
  similar to $F$.
\end{proof} 

\subsection{The full boolean matrix monoid}
\label{sec:FullBoolMat}
As mentioned above, in order to compute minimal generating sets for $\Bn$ it is
sufficient to compute sets of representatives of the prime $\J$-classes of
$\Bn$. In this section we describe how such a computation may be performed.

The \defn{kernel} of a function $f: X \to Y$ is the equivalence relation
containing a pair $(a, b)$ if and only if $f(a) = f(b)$.
We call a function $\phi: \Bn \to \Bn$ a \defn{canonical form} if
$\ker\phi = \J$. Given a canonical form $\phi$, the image $\im\phi$ is a
set of representatives of the $\J$-classes of $\Bn$, and the image
$P_\phi = \phi(\set{A \in \Bn}{\text{$A$ prime}})$ is a set of
representatives of the prime $\J$-classes of $\Bn$. 

We wish to enumerate $P_\phi$ for some canonical form $\phi$. Roughly speaking,
our strategy is to enumerate efficiently as small a superset $Q_\phi \subsetneq
\im\phi$ of $P_\phi$ as is practical, and then to filter $Q_\phi$ to remove the
non-prime matrices.
The full image $\im\phi$ is a set of representatives for the $\J$-classes of
$\Bn$; \cref{tab:BreenFormMatrices} demonstrates the growth of $\im\phi$ for
$n = 1, \ldots, 8$. The number of trim matrices in Breen form in $\Bn$ is
comparable to the number of $\J$-classes for $n \leq 8$.
Since it is infeasible to compute the image of each of the $2^{n^2}$ elements of
$\Bn$ for $n \geq 7$, we instead compute the images of a smaller set of matrices
of a particular form, which contains at least one matrix of every $\J$-class of
$\Bn$.

\begin{de}[{{\cite[Proposition 3.6]{Breen1997aa}}}]
  \label{de:BreenForm}
  We say that a matrix $A \in \Bmn$ is in \emph{Breen form} if it has all of
  the following properties:
  \begin{enumerate}[label={\rm (\roman*)}]
  \item{$A$ is reduced,}
  \item{all non-zero rows of $A$ are at the bottom,}
  \item{all non-zero columns of $A$ are at the right,}
  \item{the non-zero rows of $A$ as binary numbers are a strictly increasing
      sequence, as are the columns}
  \item{all ones in the first non-zero row of $A$ are on the right,}
  \item{all ones in the first non-zero column of $A$ are at the bottom,}
  \item{every non-zero row has at least as many ones as the first non-zero row.}
  \end{enumerate}
\end{de}

This definition appears as a proposition in~\cite{Breen1997aa}; Breen defines a
matrix in $\Bn$ to be in \emph{standard form} if the matrix has minimal value as
a binary number in its $\J$-class, and proves that such a matrix has the
properties of \cref{de:BreenForm}~\cite[Proposition 3.6]{Breen1997aa}. This leads
directly to the following proposition:

\begin{prop}
  \label{prop:BreenFormsExist}
  In every $\J$-class of $\Bn$, there exists a matrix in Breen form. 
\end{prop}

In contrast, being in Breen form is not enough to guarantee that a matrix is
minimal. Consequently there is not necessarily a unique matrix in Breen form
in any given $\J$-class of $\Bn$, as the following example demonstrates. 
\begin{ex}
  \label{ex:SimilarBreenMatrices}
Let $A, B \in \Bn$ be the matrices defined by
\begin{align*}
  A = \begin{pmatrix}
    0 & 0 & 0 & 1 & 1 \\
    0 & 0 & 1 & 0 & 1 \\
    0 & 1 & 1 & 0 & 0 \\
    1 & 0 & 0 & 1 & 0 \\
    1 & 1 & 0 & 0 & 1 
  \end{pmatrix}\quad \text{and} \quad
  B = \begin{pmatrix}
    0 & 0 & 0 & 1 & 1 \\
    0 & 0 & 1 & 0 & 1 \\
    0 & 1 & 0 & 1 & 0 \\
    1 & 0 & 1 & 0 & 0 \\
    1 & 1 & 0 & 0 & 1 
  \end{pmatrix}.&\\
\end{align*}
Then $A$ and $B$ are in the Breen form of
\cref{de:BreenForm}. Swapping rows $1$ and $2$ and columns $3$
and $4$ shows that $A$ and $B$ are similar, and hence are $\J$-related in $\Bn$.
\end{ex}

There are significantly fewer than $2^{n^2}$ trim matrices in Breen form in $\Bn$,
as shown in \cref{tab:BreenFormMatrices}, and so it is feasible to
enumerate the matrices in Breen form for several values of $n$ for which it
is not feasible to enumerate the matrices in $\Bn$.
Given the set $\mathcal{B}_n$ of matrices in Breen form in $\Bn$, and a canonical form
$\phi$, the image $Q_\phi = \phi(\mathcal{B}_n)$ contains $P_\phi$. Later in this section,
we will discuss several methods of filtering $Q_\phi$ to obtain the prime
matrices. 

\begin{table}
  \centering
  \begin{tabular}{l|r|r|r|r|r}
    $n$ & $|\Bn|$ & $|\mathcal{B}_n|$ & $|\mathcal{TB}_n|$ & $|\phi(\mathcal{TB}_n)|$ &
    $|\Bn / \J|$\\
      \hline
    $1$ & 2 & 2 & 2 & 2 & 2\\
    $2$ & 16 & 4 & 3 & 3 & 3\\
    $3$ & 512 & 13 & 5 & 5 & 11\\
    $4$ & 65\ 536 & 146 & 12 & 10 & 60\\
    $5$ & 33 554 432& 7\ 549 & 141& 32 & 877\\
    $6$ & 68\ 719\ 476\ 736& 1\ 660\ 301& 15\ 020 & 394 & 42\ 944\\
    $7$ & $5.6 \times 10^{14}$ & 1\ 396\ 234\ 450 & 7\ 876\ 125 & 34\ 014 & 7\ 339\ 704 \\
    $8$ & $1.8 \times 10^{19}$ & ? & 18\ 409\ 121\ 852 & 17\ 120\ 845 & 4\ 256\ 203\ 214
  
  \end{tabular}
  \vspace{1cm}

  \caption{The sizes of: the monoid $\Bn$, the set $\mathcal{B}_n$ of matrices
    in $\Bn$ in Breen form; the set $\mathcal{TB}_n$ of trim matrices in Breen
    form in $\Bn$; the image $\phi(\mathcal{TB}_n)$ of the trim matrices in Breen form
    under any canonical form $\phi$; and the set $\Bn / \J$ of $\J$-classes of
    $\Bn$ (see \cite{Breen2001aa, Breen1997aa} for details of the last).}
  \label{tab:BreenFormMatrices}
\end{table}

While theoretically any canonical form $\phi$ is sufficient, the complexity of
computing $\phi$ is important. We now describe how to obtain canonical forms
$\Phi_n$ that are practical to compute with, using a reduction to bipartite
graphs.

Given a matrix $A \in \Bn$, we may form the vertex-coloured bipartite graph
$\Gamma(A)$ with vertices $\{1, \ldots, 2n\}$, colours 
\[\mathbf{col}(v) = \begin{cases}
    0 \qquad &\text{if } 1 \leq v \leq n, \\
    1 \qquad &\text{if } n < v \leq 2n,
  \end{cases}
\]
and an edge from $i$ to $j+n$ if and only if $A_{ij} = 1$. The numbers $\{1,
  \ldots, n\}$ represent indices of rows and the numbers $\{n + 1, \ldots, 2n\}$
represents indices of columns in the matrix $A$.

It is easy to see that $\Gamma$ is a bijection from $\Bn$ to the set $\BGSet$ of
bipartite graphs with two parts of size $n$, where one part is coloured $0$ and
the other coloured $1$. We call a function $\psi_n: \BGSet \to \BGSet$ a
\defn{canonical form} for $\BGSet$ if the equivalence classes of $\ker\psi$ are
the graph-theoretical colour-preserving isomorphism classes of $\BGSet$.
Computing canonical forms for graphs is a well-studied problem; for a recent
article see~\cite{McKay2014aa} and the references within.
We use the software bliss\footnote{In fact, a slightly-modified version of bliss
  which avoids repeated memory allocation was used; this is available at
  \url{https://github.com/digraphs/bliss}}~\cite{Junttila2007aa, Bliss} to
compute such canonical forms $\psi_n$. Canonical forms for $\Bn$ may then be
obtained through the following lemma.
\begin{lemma}
  \label{lem:GraphCanonicalForms}
  The functions $\Phi_n = \Gamma^{-1}\psi_n\Gamma$ are canonical forms when
  restricted to reduced matrices. 
\end{lemma}
\begin{proof}
  We must show that $\ker\Phi_n = \J$, in other words $\Phi_n(A) = \Phi_n(B)$ if and
  only if $A\J B$. This is equivalent to showing $\Gamma(A)$ is isomorphic to
  $\Gamma(B)$ if and only if $A \J B$. Denote the vertices of $\Gamma(A)$ by
  $\{1, \ldots, 2 n\}$ and the vertices of $\Gamma(B)$ by $\{1', \ldots, 2n'\}$,
  as above.
  Suppose that there is a colour-preserving isomorphism $\Psi: \Gamma(A) \to
  \Gamma(B)$. Since $\Psi$ preserves colours, $\Psi$ maps $\{1, \ldots, n\} \to
  \{1', \ldots, n'\}$ and $\{n + 1, \ldots, 2n\} \to \{(n + 1)', \ldots, 2n'\}$.
  Define permutations $\alpha, \beta$ on $\{1,\ldots, n\}$ by 
  \begin{align*}
    \alpha(i) &= j \text{ if } \Psi(i) = j',\\
    \beta(i)  &= j \text{ if } \Psi(n + i) = (n + j)'.
  \end{align*}
  Then by permuting the rows and columns of $A$ by $\alpha$ and $\beta$
  respectively, $A$ is similar to $B$ and hence $A \J B$ in $\Bn$.

  Conversely, suppose that $A \J B$ in $\Bn$. Then by
  \cref{lem:PermutingReducedMatrices}, there exist $\alpha$ and $\beta$ such
  that $B$ is obtained by permuting the rows and columns of $A$ by $\alpha$ and $\beta$
  respectively. The map $\Psi$ defined by
  \[\Psi(i) = \begin{cases}
      j' \quad &\text{if } 1 \leq i \leq n \text{ and }\alpha(i) = j, \\
      (n + j)' \quad &\text{if } n < i \leq 2n \text{ and }\beta(i) = j
    \end{cases}
  \]
  is a colour-preserving isomorphism between $\Gamma(A)$ and $\Gamma(B)$.
\end{proof}

Given $v \in \B^n$, it will convenient to denote the number represented by $v$
in binary as $\num(v)$. We will also write $\vect$ for $\mathbf{num}^{-1}$.

We now describe how to backtrack through matrices in Breen form to find a
superset of prime matrices $Q_{\Phi_n}$. This may be seen as a depth-first
traversal of a (non-rooted) tree, with nodes $m\times n$ matrices ($m \leq
n$) and leaves $n \times n$ matrices.

\begin{alg}
  \label{alg:canonicalbacktrack}Backtracking for canonical forms.\\
  \textbf{Input}: A natural number $n$. \\
  \textbf{Output}: A set $Q_{\Phi_n}$, with $P_{\Phi_n} \subseteq Q_{\Phi_n}
  \subseteq \im\Phi_n$.
  \begin{enumerate}
    \item Assume that we are at a node $A$ of dimension $m \times n$, and the
      index of the first non-zero row of $A$ is $f \leq m$. 
    \item If $m = n$, the non-zero columns of $A$ form a strictly increasing
      sequence under $\num$, and $A$ is column-reduced, then add $\Phi_n(A)$ to
      $X$, the set of matrices to return.
    \item If $m < n$, then for each $x \in \{\num(A_{m*}) + 1,
        \ldots, 2^n - 1\}$, if:
      \begin{enumerate}[label={(\roman*)}]
        \item $\vect(x)$ does not contain $A_{l*}$ for any $1 \leq l \leq m$,
          and
        \item $\vect(x)$ has at least as many ones as $A_{f*}$,
          and
        \item for all column indices $1 \leq i < j \leq m$ such that $A_{*i} =
          A_{*j}$, if $\vect(x)_i = 1$ then $\vect(x)_j = 1$,
      \end{enumerate}
      then set the current node to be the matrix obtained from $A$ by adjoining
      $\vect(x)$ as the last row, and return to step 1.
    \item after every $x$ has been processed, return to the previous node (if
      any) and carry out step $3$ for the next $x$ at that node (if any).
  \end{enumerate}
  Having initialised step $1$ with each $m \times n$ matrix ($m \geq 1$)
  consisting of $m - 1$ zero rows followed by a row containing some non-zero
  number of $1$s on the right, return $X$.
\end{alg}

\begin{lemma}
  \label{lem:backtrackworks}
  The output of \cref{alg:canonicalbacktrack}, with input $n$, is a subset of
  $\im\Phi_n$ containing $P_{\Phi_n}$. Moreover, \cref{alg:canonicalbacktrack}
  does not compute $\Phi_n(A)$ of any matrix $A$ not in Breen form.
\end{lemma} 
\begin{proof}
  We will prove that the $A$s of step $4$ for which $\Phi_n(A)$ are calculated
  are precisely the set of trim matrices in Breen form; since every prime
  matrix is trim (\cref{lem:PrimeMatricesAreTrim}) and every $\J$-class
  contains a matrix in Breen form, this implies that the output contains
  $P_{\Phi_n}$. We will first prove that each such $A$ is in Breen form.

  Define a matrix to be in \emph{quasi-Breen form} if it satisfies each
  property of \cref{de:BreenForm} other than being column reduced and the
  non-zero columns forming a strictly-increasing sequence.
  We will prove that each node $A$ visited in the algorithm is trim and in
  quasi-Breen form. Note that the matrices that the algorithm is initialised
  with are all trim and in quasi-Breen form, so we must simply prove that
  passing from a node $A$ which is trim and in quasi-Breen form to a node
  $A'$ by adding a row $\vect(x)$ in step $3$ preserves these properties.

  Trimness is preserved due to condition (i) of step $3$. Since all trim
  matrices are reduced, $A'$ is also reduced. The conditions on non-zero rows
  being at the bottom and forming a strictly increasing sequence of binary
  numbers are satisfied by choosing $x$ from the range $\{\num(A_{m*}) + 1,
    \ldots, 2^n - 1\}$.  The conditions on non-zero columns being on the right
  follows from requirement (iii) of step $3$; note that this requirement also
  forces the columns to appear in (not-necessarily-strictly) increasing order.
  The first non-zero column contains the most significant digit of the rows as
  binary numbers, and since the rows of $A'$ are increasing the set of rows with
  that digit equal to $1$ must be contiguous and at the end of $A'$. Hence all
  of the ones in the first non-zero column of $A'$ are at the bottom.

  Since each leaf $A$ visited is trim and in quasi-Breen form, the two
  conditions of step 2, that the non-zero columns form a strictly increasing
  sequence and that $A$ is column-reduced, guarantee that $A$ is in standard
  form. Hence, we only calculate the canonical form $\Phi_n(A)$ if $A$ is trim
  and in Breen form.

  It remains to prove that every trim matrix $A \in \Bn$ in Breen form is
  visited as a node in the enumeration. First, note that the first $m$ rows of
  any trim, standard-form matrix $A \in \Bn$ form an $m \times n$ trim matrix in
  quasi-Breen form. Also, the zero-rows of $A$ together with the first
  non-zero row form one of the matrices with which step $1$ is initialised. It
  simply remains to show that from each $m \times n$ node $X$ consisting of the
  first $m$ rows of $A$, we visit the $(m + 1)\times n$ node $X'$ consisting of
  the first $m + 1$ rows of $A$. It is easy to verify that each of the three
  conditions in step $3$ is satisfied by row $m + 1$ of $A$, and hence $X'$ is
  visited.
\end{proof}

Given $\Phi_n$ and $Q_{\Phi_n}$, the final step is to detect when an element of
$Q_{\Phi_n}$ is prime. We present two algorithms for doing so, in
\cref{alg:filter1} and \cref{alg:filter2}.

\begin{alg}
  \label{alg:filter1}Filtering canonical forms by row spaces.\\
  \textbf{Input}: A set $Q_{\Phi_n}$, containing the images $P_{\Phi_n}$ of the
  prime matrices of $\Bn$ under $\Phi_n$, and not containing any permutation
  matrices. \\
  \textbf{Output}: The set $P_{\Phi_n}$.
  \begin{enumerate}
  \item 
    Compute $X = \set{\RowS(A\alpha)}{A \in Q_{\Phi_n} \cup \{E\}, \alpha \in
      S_n}$
  \item
    For every $A \in Q_{\Phi_n}$, and for every $\RowS(B\beta) \in X$, if $A
    \neq B$ and $\RowS(A) \subsetneq \RowS(B\beta)$ then discard
    $\RowS(A\alpha)$ from $X$ for all $\alpha \in S_n$.
  \item
    Output the set of non-elementary elements $A$ such that $\RowS(A)$ remains
    in $X$ after the previous step.
  \end{enumerate}
\end{alg}

\begin{lemma}
  The output of \cref{alg:filter1} is a set of representatives of prime
  $\J$-classes.
\end{lemma}
\begin{proof}
  Since the $\J$-class of a prime matrix consists only of similar matrices, the
  set $\set{\RowS(A\alpha)}{A \in P_{\Phi_n}, \alpha \in S_n} \subset X$
  contains all row spaces of prime matrices in $\Bn$, and hence so does $X$.
  Similarly, $X$ contains the row space of every elementary matrix. Hence, the
  elements that are maximal in $X$ are precisely the maximal elements of
  $\beta_n = \set{\RowS(A)}{A\in \Bn\setminus S_n}$ and thus by
  \cref{thm:MaximalRowSpaces} correspond to primes and elementary matrices.
  Since $\RowS(A)$ remains in $X$ after step 2 precisely when $\RowS(A)$ (and
  $\RowS(A\alpha)$, for all $\alpha \in S_n$) is maximal in $X$, the output is
  $P_{\Phi_n}$.
\end{proof} 

Note that step 2 of \cref{alg:filter1} requires $O(|Q_{\Phi_n}||X|)$ comparisons.  The
size of $X$ grows extremely rapidly with $n$ as shown in
\cref{tab:filter1numbers}, and so this algorithm is only suitable for small $n$.
\begin{table}
  \centering
  \begin{tabular}{l|r|r}
    $n$ & $|Q_{\Phi_n}|$ & $|X|$ \\
    \hline
    3 & 6 & 91 \\ 
    4 & 11 & 588 \\
    5 & 33 & 8194 \\
    6 & 395 & 570\ 636 \\ 
    7 & 34\ 015 & 342\ 915\ 296 \\
    8 & 17\ 120\ 845 & ? 
  \end{tabular}
\vspace{1cm}

\caption{The number of row spaces generated during \cref{alg:filter1} when
  given input $Q_{\Phi_n}$, the output of \cref{alg:canonicalbacktrack}}. 
  \label{tab:filter1numbers}
\end{table}

Using \cref{alg:canonicalbacktrack} and \cref{alg:filter1}, minimal
generating sets for $n \leq 7$ may be obtained.

For $n=8$ \cref{alg:filter1} is no longer sufficient, and it is necessary to use
heuristics to reduce the size of the input set $Q_{\Phi_n}$. The simplest way to
do this is to select a small subset $Y \subset Q_{\Phi_n}$ of matrices with
large row spaces, generate the row spaces $Z = \set{\RowS(A\alpha)}{A \in Y,
  \alpha \in S_n} \subset X$, and check whether $\RowS(B)$ is properly contained
in any element of $Z$ for each element $B \in Q_{\Phi_n}$.  It is also
worthwhile to filter $Q_{\Phi_n}$ by checking containment in the row spaces of
all the column permutations of a known set of prime matrices, such as those
described in the following lemma.

\begin{lemma}
  \label{lem:ExtendingPrimeMatrices}
  Let $A$ be a prime matrix in $M_{n-1}(\B)$. Extend $A$ to a matrix $B \in \Bn$
  by adding a row of zeros at the bottom and a column of zeros on the right,
  then setting $B_{n,n} = 1$. Then $B$ is prime in $\Bn$.
\end{lemma}
\begin{proof}
  Let $C$ be a matrix with row space maximal in $\beta_n$, such that $\RowS(B)
  \subseteq \RowS(B)$.  The last row of $B$ must also be in $C$ since it is a
  minimal in $\B^{n}$. Now $\RowS(A)$ has a unique basis of $n-1$ rows which
  must be contained in both $B$ and $C$, so we have $\RowB(B) = \RowB(C)$; hence
  $\RowS(B)$ is maximal in $\beta_n$. Since $B$ is not elementary, it is prime.
\end{proof}

For $n=8$ prefiltering by some of the large row spaces and extended prime
matrices is enough to obtain a minimal generating set, although the computation
is extremely lengthy; see \cref{tab:runtimestats} for some details. A
significant improvement can be obtained by using Zaretskii's Theorem.

\begin{table}
  \centering
\begin{adjustbox}{width=1\textwidth}
  \begin{tabular}{l|r|r|r|r|r|r}
    $n$ & \cref{alg:canonicalbacktrack} & \cref{alg:filter1} &
    \cref{alg:filter2} & prefiltering & \thead{\cref{alg:filter1} with \\
      prefiltering}  & \thead{\cref{alg:filter2} with \\ prefiltering} \\
    \hline
    3 &5.4ms  & 12ms & 6.5s & 23ms & 11ms & 6.5s \\
    4 &5.6ms  & 13ms & 7.0s & 25ms & 11ms & 6.7s \\
    5 &9.2ms  & 16ms & 7.2s & 33ms & 16ms & 7.3s \\
    6 &150ms  & 216ms & 8.6s & 80ms & 171ms & 8.7s \\
    7 &55s    & 1h\ 20m & 1m\ 11s & 28s & 1h\ 9m & 1m\ 8s \\
    8 &30h\ 52m & - & - & 5h\ 30m & approx. 30 days & 3.9 days
  \end{tabular}
\end{adjustbox}

\caption{Runtimes of the algorithms described in this section, for $3 \leq n
  \leq 8$. Runtimes for $n = 1,2$ were excluded as they are essentially
  immediate after accounting for system factors. Matrices were prefiltered by
  generators of $\Bm{n - 1}$ extended as in
  \cref{lem:ExtendingPrimeMatrices}; in particular a good balance was found
  to be taking those extended matrices with the 13 largest row spaces amongst
  the extended row spaces. Each algorithm was executed on a cluster of 60 2.3GHz
  AMD Opteron 6376 cores with 512GB of ram. \cref{alg:filter2} involves
  launching GAP multiple times, which took approximately 6s in each run. For
  $n=8$, \cref{alg:filter1} with prefiltering was run on a machine with 32 AMD
  Opteron 6276 cores and 192GB of RAM}
\label{tab:runtimestats}
\end{table}

Let $A \in \Bn$. The \defn{graph} of $\RowS(A)$ is the directed graph with
vertices $\RowS(A)$ and an edge from $v$ to $w$ if and only if $v \leq w$.

It is an immediate corollary of Zaretskii's Theorem that $J_A \leq J_B$ if and
only if there exists a homomorphic embedding of the graph of $\RowS(A)$ into the
graph of $\RowS(B)$ which respects non-edges. An efficient and optimised search
for such embeddings is implemented in~\cite{Digraphs2020aa}.

For practical computational purposes, it is useful to add extra structure to the
row space graphs to guide searches for embeddings. The \emph{augmented graph of
  $\RowS(A)$} is the disjoint union of the graph of $\RowS(A)$ with the empty
graph on the vertices $C = \{c_i \: : \: 1 \leq i \leq n\}$, with an edge from
$v \in \RowS(A)$ to $c_i$ if and only if $v_i = 1$.

\begin{lemma}
  \label{lem:EmbeddingGraphs}
  Let $Q$ be a superset of a canonical set of prime matrices $P$ which does not
  contain a permutation matrix, and let $A \in
  Q$. Then $A$ is not prime or elementary if and only if for some $B \in
  (Q\cup\{E\})\setminus\{A\}$ there exists a digraph embedding $\phi$ from the
  augmented graph of $\RowS(A)$ into the augmented graph of $\RowS(B)$ which
  permutes $\{ c_i : 1 \leq i \leq n \}$ and respects non-adjacency.
\end{lemma}
\begin{proof}
  Let $A$ not be prime or elementary; then $\RowS(A)$ is contained in the row
  space of some column permutation $\alpha$ of some $B \in Q \setminus \{A\}$.
  Then the embedding that extends the map $c_i \to c_{\alpha^{-1}i}$ has the
  properties required. Conversely, if such a map $\phi$ exists, the permutation
  $\alpha$ induced by the restriction to $C$ has the property that
  $\RowS(B\alpha^{-1})$ contains $\RowS(A)$, and hence by
  \cref{thm:MaximalRowSpaces} $A$ is not prime or elementary. 
\end{proof}

Note that such an embedding $\phi$ must also map a vector containing $i$ ones to
another containing $i$ ones in order to preserve adjacency and non-adjacency
with the set $C$.

We can therefore use the following improved algorithm to filter canonical
supersets of prime matrices.

\begin{alg}
  \label{alg:filter2} Filtering canonical forms by digraph embeddings.\\
  \textbf{Input}: A set $Q_{\Phi_n}$, containing the images $P_{\Phi_n}$ of the
  prime matrices of $\Bn$ under $\Phi_n$, and not containing any permutation
  matrices. \\
  \textbf{Output}: The set $P_{\Phi_n}$.
  \begin{enumerate}
  \item
    Generate the set $G$ of augmented graphs of row spaces of matrices in $Q$.
  \item 
    For every $K, L \in G$, if there exists an embedding of $K$ into $L$ as in
    \cref{lem:EmbeddingGraphs}, then discard $K$ from $G$.
  \item
    Output $X\subset Q$, the set of non-elementary elements $A$ with
    corresponding graphs remaining in $G$ after the previous step.
\end{enumerate}
\end{alg}

Note that this is in effect the same computation as in \cref{alg:filter1}; it
replaces a brute-force search through all permutations of columns with a guided
search for an appropriate permutation. It is also superior in that no more data
has to be computed, unlike \cref{alg:filter1} where new row spaces must be
produced for each column permutation. Additionally, information about the graphs
can be reused (in particular their automorphism group). However, this is not as
useful as it might seem, since the automorphism group of prime row spaces
appears to almost always be trivial.

Using this method of filtration, a minimal generating sets for $\Bn$, $6 \leq n
\leq 8$ have been computed; the size of such generating sets is contained in
\cref{tab:BMatResults}. It seems unlikely that these methods can produce
minimal generating sets for $n > 8$. The generating sets obtained from this
algorithm are obtainable at~\cite{Results2020aa}, along with code to produce
them.

\subsection{Reflexive boolean matrices}
\label{sec:RefBoolMat}
An interesting submonoid of $\Bn$ is the monoid $\Refln$ of reflexive boolean
matrices, that is, boolean matrices with an all-$1$ main diagonal. Minimal
generating sets for $\Refln$ are significantly larger than those of $\Bn$.
\begin{thm}
  \label{thm:ReflexiveGenSet}
  The unique minimal monoid generating set for $\Refln$ consists of the set of
  elementary matrices in $\Refln$ together with the set of indecomposable trim
  matrices in $\Refln$.
\end{thm}

In order to prove this theorem, we must understand the  $\J$-relation on
$\Refln$. Let $A, B$ be two matrices belonging to $\Refln$. Observe that since $A,
B$ both contain the identity matrix, the product $AB$ must contain both every
row of $A$ and every column of $B$; hence $A \leq AB$ and $B \leq AB$. This
leads to the following well-known lemma:

\begin{lemma}
  \label{lem:ReflexiveJRelation}
  $\Refln$ is $\J$-trivial. 
\end{lemma}
\begin{proof}
  If $A \J B$, then $A \leq B$ and $B \leq A$ with respect to containment. Hence
  $A = B$.    
\end{proof}
Note that it also follows that $AB$ has at least as many $1$s as the maximum
number of $1$s in $A$ or $B$. It also follows from the lemma that any
decomposable element $A$ is decomposable into a product of elements not
$\J$-related to $A$.

We now prove several results relating to decomposability of elements in $\Refln$.

\begin{lemma}
  \label{lem:ReflexiveNonTrimDecomposable}
  Every matrix in $\Refln$ that is neither trim nor elementary is decomposable in
  $\Refln$.
\end{lemma}
\begin{proof}
  Let $A \in \Refln$ be neither trim nor elementary.
  Since $A$ is not trim, it is either not row-trim or not column-trim (or both).
  Suppose that $A$ is not row-trim; then there exist some distinct $1 \leq i, j
  \leq n$ such that the $i$th row $A_{i*}$ is contained in the $j$th row
  $A_{j*}$.
  Let $B$ be the matrix obtained by setting entry $i$ of row $j$ of $A$ to be
  equal to $0$, that is
  \[B_{kl} = \begin{cases} 
                0 \quad& k = j \text{ and } l = i \\ 
                A_{kl} \quad& \text{otherwise.} 
              \end{cases} \]
  Since $i \neq j$ and $A$ is reflexive, so too is $B$. Now $B_{j*} \cup A_{i*}
  = A_{j*}$, and so $A = E^{j,i}B$. By \cref{lem:ReflexiveJRelation} neither
  $E^{j,i}$ nor $B$ is $\J$-related to $A$.
  
  If, instead, $A$ is column trim, then the same argument applied to the
  transpose $A^T$ demonstrates that $A^T$ may decomposed in such a way, and thus
  $A$ may also.
\end{proof}

\begin{cor}
  \label{cor:reflexiveindecomposable}
  An indecomposable element of $\Refln$ is either trim or elementary.
\end{cor}

\begin{lemma}
  \label{lem:reflexiveelementaryindecomposable}
  Elementary matrices are indecomposable in $\Refln$.
\end{lemma}
\begin{proof}
  Elementary matrices are precisely those matrices in $\Refln$ containing $n + 1$
  ones. As noted above, the product of two matrices in $\Refln$ contains at least
  as many $1$s as the maximum number of $1$s in a factor. It follows that if an
  elementary matrix is decomposable it is decomposable into other elementary
  matrices. However, it is routine to verify that the number of $1$s in a
  product of any two distinct elementary matrices is at least $n + 2$, and that
  each reflexive elementary matrix is idempotent. Hence elementary matrices are
  indecomposable.
\end{proof}

\begin{proof}[Proof of \cref{thm:ReflexiveGenSet}]
  Let $T$ denote the set of indecomposable trim matrices, and $\mathcal{E}$
  denote the set of reflexive elementary matrices. By
  \cref{lem:ReflexiveJRelation} and \cref{cor:reflexiveindecomposable}, it follows
  from \cref{lem:GenSetDecomposition} that $\genset{T \cup \mathcal{E}} =
  \Refln$. Since each element of $T \cup \mathcal{E}$ is indecomposable
  by definition or by \cref{lem:reflexiveelementaryindecomposable},
  $T \cup \mathcal{E}$ must be contained in any generating set for $\Refln$, and hence is
  minimal.
\end{proof}

We now discuss how we may compute minimal generating sets for $\Refln$. Since it
is easy to enumerate the reflexive elementary matrices, the problem is
determining the set of indecomposable trim matrices in $\Refln$. The following
lemma gives a method for testing indecomposability.

\begin{lemma}
  \label{lem:DecomposeIntersection}
  A trim matrix $A \in \Refln$ is decomposable if and only if it may be written
  as a product $A = BC$ of matrices $B, C \in \Refln\setminus\{I, A\}$ where the
  rows of $C$ are intersections of rows of $A$, and $B$ is the greedy left
  multiplier of $(C, A)$.
\end{lemma}
\begin{proof}
  The reverse direction is immediate. For the forward direction, suppose that
  $A$ is decomposable; that is, there exist matrices $Y, Z \in
  \Refln\setminus\{I, A\}$ such that $A = YZ$. For $1 \leq i \leq n$, define
  $K_i$ to be the set of positions of $1$s in the $i$th column of $Y$, and
  define $C$ to be the matrix with rows $C_{i*} = \cap_{j \in K_i} A_{j*}$.
  Now since $A = YZ$, for all $j \in K_i$ we have $Z_{i*} \leq A_{j*}$. It
  follows that $Z_{i*} \leq C_{i*} \leq A_{j*}$ for all $j \in K_i$. The row
  $(YC)_{i*}$ consists of the union of those rows $C_{j*}$ such that $Y_{ij} =
  1$. The condition that $Y_{ij} = 1$ is equivalent to $i \in K_j$; note the
  change in indices from previous uses of $K$. For each such $j$, we have
  $C_{j*} \leq A_{i*}$ by the inequality above. It follows that $YC \leq A$. But
  from the same inequality above, we then have $A = YZ \leq YC \leq A$; hence $A
  = YC$. By \cref{lem:GreedyMultipliers}, $A = BC$ where $B$ is the greedy left
  multiplier of $(C, A)$. It remains to show that $B, C \in \Refln\setminus\{I,
    A\}$. Since the greedy left multiplier $B$ is the maximal (by containment)
  matrix whose product with $C$ is $A$, we have $Y \leq B$ and hence $B$ is not
  the identity matrix. Similarly, $Z \leq C$ and so $C$ is not the identity
  matrix. Since $A$ is trim, if either $B$ or $C$ were equal to $A$ then the
  product $BC$ would have more $1$s than $A$ does, but $A = BC$.
\end{proof}

In order to determine whether a trim matrix $A$ is decomposable it therefore
suffices to:
\begin{enumerate}
  \item generate all matrices $C$ whose rows are intersections of rows of $A$,
    and for each $C$
  \item test whether the product $BC$ of $C$ with the greedy left multiplier
    $B$ of $(A, C)$ is equal to $A$, and $B, C \not\in \{I, A\}$.
\end{enumerate}
If no such matrices $B, C$ are found, then $A$ is indecomposable. This method
avoids computing unnecessary products, but still quickly becomes infeasible to
use for all trim matrices in $\Refln$.\\

In order to reduce the time spent checking matrices using
\cref{lem:DecomposeIntersection}, we would like to only check the smallest
necessary set of representatives of some equivalence relation. The most obvious
choice is to take canonical representatives in $\Bn$. However, the following
example illustrates that two reduced reflexive matrices can belong to the same
$\J$-class of $\Bn$ while only one of them is decomposable in $\Refln$; hence we
can not use the same canonical forms $\Phi_n$ as were used in the algorithms in
\cref{sec:FullBoolMat}.
\begin{ex}
  Let
\begin{align*}
  A = \begin{pmatrix}
    1 & 0 & 0 & 0 & 1 \\
    0 & 1 & 0 & 1 & 0 \\
    0 & 0 & 1 & 1 & 1 \\
    1 & 0 & 1 & 1 & 0 \\
    0 & 1 & 1 & 0 & 1 
  \end{pmatrix}\qquad\text{and}\qquad
  B = \begin{pmatrix}
    1 & 0 & 1 & 0 & 0 \\
    0 & 1 & 0 & 1 & 0 \\
    0 & 0 & 1 & 1 & 1 \\
    1 & 0 & 0 & 1 & 1 \\
    0 & 1 & 1 & 0 & 1 
  \end{pmatrix}.&\\
\end{align*}
Then $A$ and $B$ belong to $\Refl{5}$, and $B$ may be obtained by
exchanging columns $3$ and $5$ of $A$. However, it can be shown that $A$ is
indecomposable whilst
\begin{align*}
  B = \begin{pmatrix}
    1 & 0 & 1 & 0 & 0 \\
    0 & 1 & 0 & 0 & 0 \\
    0 & 0 & 1 & 1 & 0 \\
    1 & 0 & 0 & 1 & 0 \\
    0 & 0 & 0 & 0 & 1 
  \end{pmatrix}
  \begin{pmatrix}
    1 & 0 & 0 & 0 & 0 \\
    0 & 1 & 0 & 1 & 0 \\
    0 & 0 & 1 & 0 & 0 \\
    0 & 0 & 0 & 1 & 1 \\
    0 & 1 & 1 & 0 & 1 
  \end{pmatrix}.&\\
\end{align*}
\end{ex} 

While the canonical forms $\Phi$ from \cref{sec:FullBoolMat} cannot be
used, the following lemma shows that it is still possible to reduce the space of
matrices that must be checked.

\begin{lemma}
  \label{lem:reflexivecanonical}
  Let $A, B \in \Refln$ be such that $A = P^{-1} B P$ for some permutation matrix
  $P \in S_n$ (i.e. $A$ is obtained by permuting the rows and columns of $B$ by
  the same permutation). Then $A$ is decomposable in $\Refln$ if and only if $B$
  is decomposable in $\Refln$.
\end{lemma}
\begin{proof}
  Suppose that $B = XY$ for $X, Y \in \Refln\setminus \{I, B\}$. Then $P^{-1}XP,
  P^{-1}YP \in \Refln\setminus\{I, A\}$ and $P^{-1}XPP^{-1}YP = A$. The proof of
  the other direction is dual, since $PAP^{-1} = B$.
\end{proof}

Hence it is sufficient to consider representatives of the equivalence
which relates any two matrices which are similar under the same row and column
permutation. In order to compute these representatives, we modify the
construction of the bipartite graphs $\Gamma$ of \cref{sec:FullBoolMat}.

Given a matrix $A \in \Refln$, we form the vertex-coloured tripartite graph
$\Gamma_\text{id}(A)$ with vertices $\{1, \ldots, 3n\}$, colours 
\[\mathbf{col}(v) = \begin{cases}
    0 \qquad &\text{if } 1 \leq v \leq n, \\
    1 \qquad &\text{if } n < v \leq 2n, \\
    2 \qquad &\text{if } 2n < v \leq 3n,
  \end{cases}
\]
an edge from $i$ to $j+n$ if and only if $A_{ij} = 1$ for $1 \leq i \leq n$,
and an edge from $i + 2n$ to $i$ and $i + n$ for $1 \leq i \leq n$. The numbers $\{1,
  \ldots, n\}$ represent indices of rows, and the numbers $\{n + 1, \ldots, 2n\}$
represents indices of columns in the matrix $A$. The additional vertices $\{2n +
  1, \ldots, 3n\}$, adjacent to both the corresponding row and column nodes,
force an isomorphism of $\Gamma_\text{id}(A)$ to induce the same permutation on rows and
columns of $A$ in the same way that a permutation was induced in
\cref{lem:GraphCanonicalForms}. As before, we may obtain canonical forms
$\Psi_n$ for the graphs $\Gamma_{\text{id}}(A)$ via bliss. Since
$\Gamma_\text{id}$ is clearly injective, we may compute the functions
$\Xi_n = \Gamma_\text{id}\Psi_n\Gamma_\text{id}^{-1}$. It is easy to show that
the equivalence classes $\ker\Xi$ are precisely the classes of matrices which
are similar under permuting rows and columns by the same permutation.

As in \cref{sec:FullBoolMat}, we wish only to enumerate matrices of a
certain form. A similar argument to~\cite[Proposition~3.6]{Breen1997aa} shows
that by permuting the rows and columns of matrices by the same permutation,
matrices in $\Refln$ can be put in the following \defn{reflexive Breen form}:

\begin{enumerate}[label={\rm (\roman*)}]
  \item{all $1$s in the first row of $A$ are on the left,}
  \item{no row has fewer ones than the first row,}
  \item{for each row $A_{i*}$, if $A_{ij} = 1$ then for each $l \in \{1, \ldots,
        j\}$ there exists $k \in \{1, \ldots, i\}$ such that $A_{kl} = 1$ .}
\end{enumerate}

As in \cref{ex:SimilarBreenMatrices}, it is possible for two distinct matrices
in this form to have the same value under $\Xi_n$.

Now, given $\Xi_n$, a similar backtrack search to \cref{alg:canonicalbacktrack}
allows appropriate representatives of matrices to be enumerated.

We then have the following algorithm for finding a minimal generating set for
$\Refln$:
\begin{alg}
  \label{alg:ReflexiveGenSet}
  Computing the minimal generating set for $\Refln$ \\
  \textbf{Input}: A natural number $n$.\\
  \textbf{Output}: A minimal generating set for $\Refln$.
\begin{enumerate}
  \item
    Enumerate the trim reflexive boolean matrices in reflexive Breen form using
    the analogue of \cref{alg:canonicalbacktrack}, storing canonical
    representatives under a row and column permutation in a set $S$.
  \item 
    Filter out the decomposable matrices in $S$ using
    \cref{lem:DecomposeIntersection}, leaving a set $T$ of trim matrices that
    are not decomposable.
  \item
    Return $T$ together with the reflexive elementary matrices.
\end{enumerate}
\end{alg}
Using this algorithm, we can calculate the sizes of minimal generating sets up
to $n = 6$; these are contained in \cref{tab:BMatResults}.

There are a small number of matrices in $\Refl{7}$ for which the approach based on
\cref{lem:DecomposeIntersection} is too inefficient; we chose to use a
different method to test $12$ matrices in total. Observe that a matrix $A$ is
decomposable into a product of generators $X_1 X_2 \ldots X_k$ from some minimal
generating set if and only if $\alpha A \alpha^{-1} = \alpha (X_1 X_2 \ldots
X_{k-1})\alpha^{-1}\alpha X_k \alpha^{-1}$ for all permutations $\alpha$. Since
the set $S$ of \cref{alg:ReflexiveGenSet} contains $\alpha X_k
\alpha^{-1}$ for some $\alpha \in S_n$, we can detect if $A$ is decomposable by
testing whether $\alpha A \alpha^{-1} = CB$ for any $\alpha \in S_n$, $B \in S$
and with $C$ the greedy left multiplier of $(\alpha A \alpha^{-1}, B)$. A brute
force approach based on this observation is sufficient to filter the $12$
difficult matrices for $n = 7$. Although this suggests that the brute-force
method is superior to that of \cref{lem:DecomposeIntersection}, this is not in
practice the case; for most matrices the approach based on
\cref{lem:DecomposeIntersection} is far more efficient. The $12$ matrices that
are particularly non-susceptible to this method for $n = 7$ have many more
combinations of intersections of rows than the other matrices (on average,
roughly $10\ 000$ times more).


\subsection{Hall matrices}
\label{sec:HallBoolMat}

The \emph{Hall monoid} is the submonoid of $\Bn$ consisting of matrices which
contain a permutation matrix. These matrices correspond to instances of the Hall
marriage problem that have a solution, and are thus referred to as \emph{Hall
  matrices}; see~\cite{Schwarz1973aa, Butler1974aa, Tan2000aa, Cho1993ab} for
further reading.
We will denote the Hall monoid by $\Halln$. For convenience, we shall often
simply say that a Hall matrix contains a permutation when it contains the
corresponding permutation matrix.

The main result of this section is the following theorem. Unlike the monoid of
reflexive boolean matrices $\Refln$, the Hall monoid $\Halln$ has minimal
generating sets that are strongly related to the minimal generating sets for
$\Bn$.
\begin{thm}
  \label{thm:HallGenSet}
  Every minimal generating set for $\Halln$ is obtained by removing a
  matrix similar to $F$ from a minimal generating set for $\Bn$. That is, every
  minimal generating set for $\Halln$ consists of a set of representatives $P$
  of the prime $\J$-classes of $\Bn$ together with a minimal set of generators
  for the group of units and an elementary matrix.
\end{thm}

In order to prove this theorem, we will need the following classical result,
restated in our context.

\begin{thm}[Hall's Marriage Theorem~{\cite[Theorem 1]{Hall1935aa}}]
  \label{thm:HallMarriage}
  Let $A \in \Bn$. Then $A$ is a Hall matrix if and only if every union of $k$
  rows contains at least $k$ ones, for $1 \leq k \leq n$.
\end{thm}
We shall say that a subset $X$ of the rows of a matrix \emph{satisfies the Hall
  condition} if the union of the rows in $X$ contains at least $|X|$ ones, and
that a matrix satisfies the Hall condition if every subset of the rows satisfies
the Hall condition.

Similarly to the case for reflexive matrices, the fact that Hall matrices contain
permutations gives useful information on products of matrices. Given two
matrices $A, B \in \Halln$, both $A$ and $B$ contain a permutation matrix; hence
$AB$ contains a row-permuted copy of $B$ and a column-permuted copy of $A$. It
follows that $AB$ contains at least as many $1$s as the maximum number of $1$s
in $A$ or $B$.

The $\J$-relation for $\Halln$ is easily described by the following lemma.
\begin{lemma}[{\cite[Theorem 2]{Butler1974aa}}]
  \label{lem:HallJRelation}
  Two matrices $A, B \in \Halln$ are $\J$-related in $\Halln$ if and only if
  they are similar.
\end{lemma}
\begin{proof}
  The reverse direction is clear. For the forward direction, suppose that $A =
  SBT$ and $B = UAV$ for matrices $S, T, U, V \in \Halln$. Then $BT$ contains a
  column-permuted copy of $B$, and hence $A$ contains a row- and column-permuted
  copy of $B$. Since, similarly, $B$ contains a row- and column-permuted copy of
  $A$, it follows that $A$ is a row- and column-permutation of $B$.
\end{proof}

We will prove \cref{thm:HallGenSet} through a series of lemmas. The
first thing to prove is that every element specified actually belongs to
$\Halln$; this is clear for elements of the group of units and elementary
matrices but less clear for prime matrices.

\begin{lemma}
  \label{lem:PrimeAreHall}
Every prime matrix is Hall.
\end{lemma}
\begin{proof}
  Let $A \in \Bn$ be prime. By \cref{thm:HallMarriage}, it suffices to show
  that there is no subset of rows of $A$ of size $k$ such that the union of
  these rows contains fewer than $k$ ones. This is equivalent to not containing
  any $k \times (n - k + 1)$ submatrix containing only $0$. As a consequence of
  the discussion after~\cite[Definition 2.4]{Caen1981aa}, the set of rows of $A$
  containing more than a single $1$ do not contain any $k \times (n - k)$
  submatrix containing only $0$; hence $A$ does not contain a $k \times (n - k +
  1)$ such submatrix, and $A \in \Halln$.
\end{proof}

In order to apply \cref{lem:GenSetDecomposition}, we must be able to decompose
certain elements of $\Halln$ into products of elements that lie above the given
element in the $\J$-order. The following two lemmas are the first steps in
finding such a decomposition.

\begin{lemma}
  \label{lem:NonTrimDecomposable}
  Let $A$ be a non-trim matrix belonging to $\Halln$. Then there exist $B, C \in
  \Halln$ such that $A = CB$, and neither $B$ nor $C$ is $\J$-related to $A$.
\end{lemma}
\begin{proof}
  We may assume that $A$ contains the identity permutation, since $A$ is
  decomposable into such matrices $B$ and $C$ if and only if every matrix
  similar to $A$ is decomposable into matrices similar to $B$ and $C$. The lemma
  now follows using the same proof as in
  \cref{lem:ReflexiveNonTrimDecomposable}.
\end{proof}

It will be convenient in the following proofs to define $e_i$ to be the boolean
vector of length $n$ with a single $1$ in position $i$, for $1 \leq i \leq n$.

\begin{lemma}
  \label{lem:HallAssumptions}
  Let $A$ be a trim matrix belonging to $\Halln$ which is not prime, elementary, or a
  permutation matrix. Then there exist $B \in \Halln$ and $C \in \Bn$ such that
  $A = CB$, neither $B$ nor $C$ is $\J$-related to $A$ in $\Halln$, and $C$
  contains a $1$ in every column.
\end{lemma}
\begin{proof}
  As in \cref{lem:NonTrimDecomposable}, we may assume that $A$ contains the
  identity permutation. Since $A$ is not prime, elementary, or a permutation
  matrix, it has non-maximal row space in
  $\beta_n = \set{\RowS(A)}{A\in \Bn\setminus S_n}$.
  Let $B$ be a maximal
  non-permutation matrix in $\Bn$ whose row space contains the row space of $A$.

  Letting $C \in \Bn$ be the greedy left multiplier of $(A, B)$, we have $A =
  CB$ by \cref{lem:GreedyMultipliers}. Suppose that $C$ is $\J$-related to $A$;
  then by \cref{lem:HallJRelation} $C$ is similar to $A$. Then since $A$ is
  trim, so too is $C$. Since $B$ was chosen to be a non-permutation matrix, the
  product $CB$ contains more $1$s than $C$. But $A = CB$, and $C$ is similar to
  $A$, a contradiction. Hence $C$ is not $\J$-related to $A$. Since $\RowS(A)
  \subsetneq \RowS(B)$, $B$ is not $\J$-related to $A$ in $\Bn$ and thus is not
  $\J$-related to $A$ in $\Halln$.

  If $C$ contains a $1$ in every column, we are done. However, there is no
  reason that this must be the case.

  Suppose that the columns of $C$ indexed by $X = \{c_1, c_2, \ldots, c_k\}$ do
  not contain a $1$, and that $B$ contains the permutation $\alpha$. For $1 \leq
  i \leq k$, define $B'_{c_i*} = e_{\alpha(c_i)}$. Define the remaining rows of
  $B'$ to be the corresponding rows of $B$. Then $B'$ contains $\alpha$. Let
  $C'$ be the greedy left multiplier of $(A, B')$; then for all $1 \leq i \leq
  k$, $C'_{\alpha(c_i)c_i} = 1$ since $A_{\alpha(c_i)\alpha(c_i)} = 1$. Note
  that since the columns indexed by $X$ of $C$ did not contain a $1$, $\RowS(A)$
  is in fact a subset of the subspace of $\RowS(B)$ generated by the rows
  indexed by $\{1, \ldots, n\}\setminus X$; hence $\RowS(A) \subsetneq
  \RowS(B')$ and thus by \cref{lem:GreedyMultipliers}, $A = C'B'$. Hence we
  have found $B', C' \in \Bn$ such that $A = C'B'$, $B' \in \Halln$, and every
  column of $C'$ contains a $1$. We must finally prove that $B'$ and $C'$ are
  not $\J$-related to $A$ in $\Halln$. 
  
  Suppose that $A \J B'$; then $B'$ is similar to $A$, and hence $B'$
  is trim. Since $C$ contains a $1$ in every row, and $C'$ contains all $1$s of
  $C$ together with at least one more, $C'$ is not the identity matrix.
  As above, it follows that $C'B'$ contains more $1$s than $B'$ does; this is a
  contradiction. Hence $A$ is not $\J$-related to $B'$. Now suppose that $A\J
  C'$. Then $C'$ is similar to $A$, and thus trim.
  In order to apply the previous argument to $C'$, we must prove that $B'$ is
  not a permutation matrix. Since $B$ was not a permutation matrix, and $B'$ was
  obtained by replacing some rows of $B$ by rows $e_{\alpha(c_i)}$ containing a
  single $1$, the only way that $B'$ may be a permutation matrix is if the
  non-replaced rows of $B$ - indexed by $\{1, 2, \ldots, n\}\setminus X$ - all
  contained a single $1$.  But since $C$ contains $1$s only in those columns,
  the product $A = CB$ could then only contain $1$s in at most $n - |X|$
  columns, and thus would not be a Hall matrix.  Hence, $B'$ is not a
  permutation matrix and a similar argument to that which showed that $A$ is not
  $\J$-related to $B'$ applies to $A$ and $C'$, to show that $A$ is not
  $\J$-related to $C'$ in $\Halln$.
\end{proof}

While \cref{lem:HallAssumptions} does not provide the decomposition we
require in order to use \cref{lem:GenSetDecomposition}, it provides a
decomposition into two matrices that are close to having the correct properties.
This motivates the following definition. Given a row-trim matrix $A \in \Bn$,
define the \defn{core} $A^\circ$ of $A$ to be the submatrix of $A$ consisting of
those rows containing at least two $1$s. We say that a subset of rows of $A$
violating the Hall condition is a \defn{maximal violator of $A$} if it has
largest cardinality amongst such subsets, and that a matrix is
\defn{$k$-deficient} if $k$ is the cardinality of a maximal violator of
$A^\circ$. The following lemma shows that we are justified in considering
maximal violators of $A^\circ$ rather than $A$. 

\begin{lemma}
  \label{MaximalViolatorsCore}
  If $A$ is a row-trim matrix belonging to $\Bn$, then $A$ has a subset of rows
  violating the Hall condition if and only if $A^\circ$ has a subset of rows
  violating the Hall condition.
\end{lemma}
\begin{proof}
  Suppose $A$ has rows $X = \{r_1, r_2, \ldots, r_k\}$ containing a single $1$.
  Since $A$ is row-trim, no other row may contain a $1$ in the columns in which
  the rows in $X$ have a $1$; hence we may permute the rows and columns of $A$ so
  that the rows containing a single $1$ form a identity matrix $I_k$ as the $k
  \times k$ bottom-right block. Now $A^\circ$ forms the first $n-k$ rows, and it
  is clear that the union of a maximal violator of $A^\circ$ with the last $k$
  rows forms a maximal violator of $A$, and that restricting a maximal violator
  of $A$ to the first $n-k$ rows induces a maximal violator of $A^\circ$.
\end{proof}

We will demonstrate that we can iteratively reduce the $k$-deficiency of the
matrix $C$ in the statement of \cref{lem:HallAssumptions} until it is
$0$-deficient and therefore belongs to $\Halln$.

\begin{lemma}
  \label{lem:DeficiencyReduction}
  Suppose that $A$ is a trim matrix belonging to $\Halln$ containing the
  identity permutation, and that $A = CB$ where $B$ belongs to $\Halln$, $C
  \in \Bn$ is a $k$-deficient matrix containing a $1$ in every column, and
  neither $B$ nor $C$ are $\J$-related to $A$ in $\Halln$. Then there exist $T
  \in \Halln$ and $S \in \Bn$ such that $A = ST$, $S$ contains a $1$ in
  every column, neither $S$ nor $T$ are $\J$-related to $A$ in $\Halln$, and
  $S$ is at most $(k-1)$-deficient.
\end{lemma}
\begin{proof}
  Since $C$ is $k$-deficient, there is some maximal violator of $C^\circ$,
  indexed in $C$ by $W  = \{w_1, w_2, \ldots, w_k\} \subset \{1, \ldots, n\}$.
  Since $C$ contains a $1$ in every column, $k < n$. We will show how we can
  construct new matrices from $C$ and $B$ so that the rows indexed by $W$
  satisfy the Hall condition. It will be useful in this proof to denote the
  complement of a set $Z \subseteq \{1, 2, \ldots, n\}$ in $\{1, 2, \ldots, n\}$
  by $Z^C$. For the sake of brevity, we will also not distinguish between
  indices of rows and the rows themselves when the distinction is not important;
  the same is true of columns.  Thus we may talk about the rows $W$ rather than
  the rows indexed by $W$.

  Let $X = \{x_1, \ldots, x_l\} \subset \{1, \ldots, n\}$ denote the indices of
  those columns of $C$ that do not contain any $1$s in the rows $W$.
  By multiplying $C$ by an appropriate permutation matrix on the right, and $B$
  by the inverse of this permutation matrix on the left, we may assume that $X^C
  \subset W$, and that $X \cap W = \{x_1, \ldots, x_{m - 1}\}$, where $m = k + l
  - n + 1$. That is, the $1$s that occur in rows $W$ occur in a subset of the
  columns $W$, and the complement of this subset in $W$ is labelled by $\{x_1,
    \ldots, x_{m - 1}\}$. Consider the remaining $n - k$ column indices $\{x_m,
    \ldots, x_k\}$. 
    
  Let $D$ be the $(n-k) \times (n-k)$ submatrix of $C$ consisting of rows $W^C$
  of $C$ and columns $\{x_m, \ldots, x_k\}$. Suppose that $D$ contains maximal
  violator $V$. The rows in $V$ cannot all be rows of $C$ which contained a
  single $1$: since $A$ is trim, $C$ is row trim and hence each row of $C$
  containing a single $1$ contains a distinct $1$.  It follows that $V$ contains
  at least one row of $C^\circ$. But then the union of $W$ with the rows of
  $C^\circ$ corresponding to rows of $V$ is a maximal violator of $C^\circ$ with
  larger cardinality than $W$, a contradiction. Hence $D$ satisfies the Hall
  condition, and therefore $D$ contains a permutation matrix. This defines a
  bijection $f: W^C \to \{x_m, \ldots, x_k\}$ such that for all $i \in W^C$,
  $f(i)$ is the column containing a $1$ in row $i$ of the permutation matrix.
  For all $i \in W^C$, we define row $i$ of $S$ to be $e_{f(i)}$, and row $f(i)$
  of $T$ to be $A_{i*}$; then $(ST)_{i*} = A_{i*}$. Observe that since
  $C_{i,f(i)} = 1$, row $B_{f(i)*}$ is contained in row $A_{i*}$ and hence also
  contained in row $T_{i*}$.

  For all $j \in X^C$, we define $T_{j*} = B_{j*}$; then for all $i \in W$,
  $(CT)_{i*} = A_{i*}$. It remains to define rows $W$ of $S$ and rows $\{x_1,
    \ldots, x_{m - 1}\}$ of $T$. Since $B$ is a Hall matrix, it contains a
  permutation matrix $U$ with corresponding permutation $u$. The rows $X^C
  \cup \{x_m, \ldots, x_k\}$ of $T$ contain the corresponding rows of $U$. For
  all $j \in \{x_1 \ldots, x_k\}$, consider the cycle in the permutation $u$
  starting at $j$. If $u(j) \in W$, then define $T_{j*} = e_{u(j)}$. Otherwise,
  $u(j) \not\in W$, and we may construct the tuple $(u^1(j), u^2(j), \ldots,
  u^{q}(j), iu^{q+1}(j))$ where $u^{i}(j) \not\in W^C$ for $1 \leq i \leq q$,
  and $u^{q + 1}(j) \in W$. Then we define $T_{j*} = e_{u^{q + 1}(j)}$. By the
  assumptions on $X$, for $1 \leq i \leq q$ we have $u^{i}(j) \in \{x_m, \ldots,
    x_l\}$, and by the definitions of rows $\{x_m, \ldots, x_l\}$ of $T$, we
  have that $T_{u^{i}(q),u^{i}(q)} = 1$. For each such tuple, we modify $U$ by
  defining $U_{j*} = e_{u^{q + 1}(j)}$, and $U_{u^{i}(q)*} = e_{u^{q + 1}(j)}$
  for $1 \leq i \leq q$. Note that this modification is well-defined, and in
  particular does not depend on the order in which the modification is carried
  out, since $u$ is a permutation and thus decomposes into disjoint cycles.
  After this modification, $T$ is fully defined and contains the permutation
  matrix $U$.

  For each $j \in \{x_1, \ldots, x_{m - 1}\}$, we have $T_{j*} = e_w$ for some
  $w \in W$, and hence there is a $1$ in position $(w, j)$ of the greedy left
  multiplier of $(A, T)$. Define the rows $W$ of $S$ to be the corresponding
  rows of the greedy left multiplier of $(A, T)$. Then the union $y$ of the rows
  $W$ of $S$ contains ones in columns $\{x_1, \ldots, x_{m - 1}\}$. Note that
    since the rows of $T$ indexed by $X^C$, the set of columns that had a $1$ in
    some row of $W$, are equal to the corresponding rows of $B$, we have $S_{w*}
    \geq C_{w*}$ for all $w \in W$. Hence $y$ contains $(n - l) + m = k$ ones,
    and so the rows $W$ satisfy the Hall condition. Since every other row of $S$
    contains a single distinct $1$, it follows that $S$ is at most $(k -
    1)$-deficient.

  As noted above, $(ST)_{i*} = A_{i*}$ for all $i \in W^C$. For $i \in W$, we
  have that $S_{i*} \geq C_{i*}$; since $(CT)_{i*} = A_{i*}$ we must have
  $(ST)_{i*} \geq A_{i*}$. By the definition of the greedy left multiplier,
  $(ST)_{i*} \leq A_{i*}$ for all $i \in W$. Hence $A = ST$.
    
  We wish to now show that $S$ contains a $1$ in every column, and that $A$ is
  not $\J$-related to $S$ or to $T$. That $S$ contains a $1$ in every column
  follows from the union $y$ of the rows $W$ containing $1$s in each of the columns
  $W$, and the other rows being defined via the bijection $f$. That
  $T$ is not $\J$-related to $A$ follows from the same argument as above; since
  $A$ is trim so would be $B$, but $S$ contains at least one row with more than
  one $1$, a contradiction. In order to prove that $S$ is not $\J$-related to
  $A$, we must as above prove that $T$ is not a permutation matrix. Similarly to
  the argument above, we note that if $T$ is a permutation matrix then each row
  of $T$ belonging to $X^C$ would contain a single one, but these are equal to the
  corresponding rows of $B$. Hence the union of rows $W$ of $CB$ would contain
  precisely $n - l < k$ ones, and thus $A$ would not be Hall. It follows that
  $T$ is not a permutation matrix, and the argument proceeds as above. 
\end{proof}

\begin{cor}
  \label{cor:HallDecomposition}
  Let $A$ be a matrix belonging to $\Halln$ which is not prime, elementary,
  or a permutation matrix. Then $A$ can be decomposed as a product $A = CB$
  of matrices $B, C \in \Halln$, such that neither $B$ nor $C$ is $\J$-related
  to $A$ in $\Halln$.
\end{cor}
\begin{proof}
  We may assume that $A$ contains the identity permutation, since the lemma
  holds for $A$ if and only if it holds for all matrices similar to $A$. If $A$
  is not trim then this follows directly from \cref{lem:NonTrimDecomposable}.
  Otherwise, $A$ is trim and by \cref{lem:HallAssumptions} there exist $B \in
  \Halln$ and $C \in \Bn$ such that $A = CB$, neither $B$ nor $C$ is
  $\J$-related to $A$ in $\Halln$, and $C$ contains a $1$ in every column. If
  $C$ also belongs to $\Halln$, then we are done; otherwise $C$ is $k$-deficient
  for some $1 \leq k \leq n$. By iteratively replacing $B$ and $C$ with the
  matrices constructed in \cref{lem:DeficiencyReduction}, we must eventually
  construct $B$ and $C$ with the properties above and where $C$ is
  $0$-deficient, and thus Hall by \cref{MaximalViolatorsCore}.
\end{proof}

We may now prove \cref{thm:HallGenSet}.

\begin{proof}[Proof of \cref{thm:HallGenSet}]
  Let $G \subset \Halln$ consist of a set of representatives $P$ of the prime
  $\J$-classes of $\Bn$ together with two matrices that generate $S_n$ and an
  elementary matrix. By \cref{lem:HallJRelation}, for any element $x \in G$ with
  $\J$-class $J_x$ in $\Halln$, we have $J_x \subsetneq \genset{G}$. It
  follows from \cref{cor:HallDecomposition} and \cref{lem:GenSetDecomposition}
  that $G$ generates $\Halln$. Since any generating set for $\Bn$ requires a
  matrix from every prime $\J$-class, two generators for $S_n$, and an
  elementary matrix, and each of these lie in $\Halln$, we must require these
  matrices in any generating set for $\Halln$; hence $G$ is minimal. The same
  argument shows that every minimal generating set is obtained in this way.
\end{proof}

\subsection{Triangular boolean matrices}
\label{sec:TriBoolMat}
We now turn our attention to the monoids of upper and lower triangular boolean
matrices, denoted by $\UTn$ and $\LTn$ respectively. These matrices have a
particularly simple minimal generating set:

\begin{lemma}
  \label{lem:mingeneratingsetfortriangular}
  The unique minimal generating set $G_U$ for $\UTn$ consists of those
  elementary matrices and matrices similar to $F$ which are upper triangular,
  together with the identity matrix. The dual statement holds for $\LTn$.
\end{lemma}

\begin{proof}
  We will prove the lemma for $\UTn$; the result for $\LTn$ follows via the
  anti-isomorphism given by transposition.
  We will first construct any matrix $A \in \UTn$ as a product of matrices in
  $G_U$. We iteratively define a product $A(i)$, $0 \leq i \leq n$ where $A(0)$
  is the identity and $n$ is the number of $1$s in $A$. For the $i$th $1$
  contained in $A$ (ordered by row then column), in row number $x_i$ and column
  number $y_i$, we obtain $A(i)$ from $A(i - 1)$ by left-multiplying by $E^{x_i,
    y_i}$. Let the zero rows of $A$ have indices $\{z_1, \ldots, z_{k}\}$. Note
  that the matrices in $\UTn$ similar to $F$ are precisely those matrices
  obtained by deleting a single $1$ from an identity matrix. We define $A(n +
  j)$ to be the matrix obtained by left multiplying $A(n + j - 1)$ by the
  element of $X$ with a zero row in position $z_j$. Then $A(n + k) = A$.
  
  It is routine to show that for each matrix $A \in G_U$, if $A$ is written as
  a product $A = BC$ in $\UTn$ then one of $B$ or $C$ must be equal to $A$; it
  follows that any generating set for $\UTn$ must contain all matrices in $G_U$.
  Hence $G_U$ is the unique minimal generating set for $\UTn$.
\end{proof}

\begin{cor}
  \label{cor:RankTriBoolMat}
  For $n \geq 2$, the ranks of $\UTn$ and $\LTn$ are $T_n + 1$, where $T_n$ is
  the $n$th triangular number $T_n = n(n+1)/2$.
\end{cor}
\begin{proof}
  Elementary matrices in $\UTn$ are precisely those in the form of an identity
  matrix together with an additional $1$ in some position above the main
  diagonal; hence there are $T_{n - 1}$ elementary matrices in $\UTn$. There are
  $n$ matrices similar to $F$ in $\UTn$, and hence $T_{n - 1} + n =  T_n$
  non-identity elements of the minimal generating set.
\end{proof}


\section{Tropical matrices}
\label{sec:Tropical}
Recall that $\Kmin$ denotes the min-plus semiring where addition is $\min$ and
multiplication is $+$, and the min-plus semiring with threshold $\Kmint$ is the
quotient of $\Kmin$ by the least congruence containing $(t, t+ 1)$. Similarly,
$\Kmax$ denotes the max-plus semiring and $\Kmaxt$ is the quotient of $\Kmax$ by
the least congruence containing $(t, t + 1)$. Generating
sets for $M_2(\Kmin)$, $M_2(\Kmint)$, $M_2(\Kmax)$ and $M_2(\Kmaxt)$ were found
in~\cite{East2020aa}; see also \cref{thm:MinPlusGenSet} and
\cref{cor:FiniteMinPlusGenSet}. In this section, we show that the generating
sets from~\cite{East2020aa} are, in fact, minimal. Minimal generating sets are
not known for such matrices of arbitrary dimension, and results in this
direction seem unlikely. For instance, if $t = 0$, then $M_n(\Kmint)$ is
isomorphic to the full boolean matrix monoid $\Bn$ via the semiring isomorphism
$\phi: \Kmint \to \B$ which maps $\phi(\infty) = 0$ and $\phi(0) = 1$. Finding
minimal generating sets for $\Bn$ is the subject of
\cref{sec:FullBoolMat}, and it seems reasonable to say that this is
somewhat difficult. In particular, any general theorem on minimal generating
sets for $M_n(\Kmint)$ must include Devadze's Theorem (\cref{thm:Devadzefull})
as a special case, and so it seems likely that any such theorem would not
provide explicit generators but rather would require the computation of those
$\J$-classes immediately below the group of units as in
\cref{sec:FullBoolMat}.

\subsection{Min-plus matrices}

\begin{thm}[{\cite[Theorem 1.1]{East2020aa}}]\label{thm:MinPlusGenSet}
  The monoid $M_{2}(\Kmin)$ of $2 \times 2$ min-plus matrices is
  generated by the matrices:
  \begin{equation*}
          A(i) = \mat{i}{0}{0}{\infty},
    \quad B     = \mat{1}{\infty}{\infty}{0},
    \quad \text{and}
    \quad C     =  \mat{\infty}{\infty}{\infty}{0}
  \end{equation*}
  where $i \in \N \cup \{\infty\}$.
\end{thm}

\begin{cor}[{\cite[Corollary 1.2]{East2020aa}}]\label{cor:FiniteMinPlusGenSet}
  Let $t \in \Np$ be arbitrary. Then the finite monoid $M_{2}(\Kmint)$ of
  $2 \times 2$ min-plus matrices is generated by the $t + 4$ matrices:
  \begin{equation*}
          A(i) = \mat{i}{0}{0}{\infty},
    \quad B     = \mat{1}{\infty}{\infty}{0},
    \quad \text{and}
    \quad C     =  \mat{\infty}{\infty}{\infty}{0}
  \end{equation*}
  where $i \in \Kmint$.
\end{cor}

\begin{thm}\label{thm:MinPlusMinimality}
  The generating sets of \cref{thm:MinPlusGenSet} and
  \cref{cor:FiniteMinPlusGenSet} are irredundant and minimal.
\end{thm}
We will prove in the following sequence of lemmas that the given generating sets
are irredundant. In the infinite case, the cardinality of the given generating
set then guarantees that it is minimal. In the finite case, it suffices, by
\cref{prop:Wilf}, to show that the generating set contains at most one element
from each $\J$-class of the monoid. 

\begin{lemma}\label{lem:MinPlusSimilarGenerators}
  Let $\S \in \{\Kmin\} \cup \set{\Kmint}{t \in \Np}$. Then any generating set
  for $M_{2}(\S)$ contains a matrix similar to $B$ and a matrix similar to
  $A(i)$, for each $i \in \S$.
\end{lemma}
\begin{proof}
It is easy to show that if a matrix $X \in M_n(\S)$ satisfies the property that
when written as a product $X = YZ$ for $Y, Z \in M_n(S)$, one of $Y$ or $Z$ must
be similar to $X$, then the same property holds for all matrices similar to $X$.
It follows that we may prove the lemma by establishing this property for $B$ and
$A(i)$, since any product of matrices equal to one of $A(i)$ or $B$ must contain
a matrix similar to that generator.

Fix $i \in \N\cup\infty$, and suppose that $A(i)$ is written as the product of
two matrices in $M_{2}(R)$. That is, suppose that 
\[\mat{i}{0}{0}{\infty} = \mat{a}{b}{c}{d} \mat{u}{v}{w}{x}\] 
where $a,b,c,d,u,v,w,x \in R$. It follows by
the definitions of matrix multiplication, $\oplus$, and $\otimes$ that
\begin{enumerate}[label=(\theenumi)]
  \item
    ($a + u = i$ and $b + w \geq i$) or ($b + w = i$ and $a + u \geq i$),
  \item
    $a = v = 0$ or $b = x = 0$,
  \item
    $c = u = 0$ or $d = w = 0$, and
  \item
    ($c = \infty$ or $v = \infty$) and ($d = \infty$ or $x = \infty$).
\end{enumerate}
Suppose that $(2)$ is satisfied by $a = v = 0$. Then $c = \infty$ by $(4)$.  It
follows that $d = w = 0$ by $(3)$, and so $x = \infty$ by $(4)$.  Thus by $(1)$,
either $u = i$ and $b \geq i$, or $b = i$ and $u \geq i$, that is,\
$$\mat{u}{v}{w}{x} = \mat{i}{0}{0}{\infty} \quad \text{or} \quad
\mat{a}{b}{c}{d} = \mat{0}{i}{\infty}{0}.$$
Instead, if we suppose that $(2)$ is satisfied by $b = x = 0$, then either
$$\mat{u}{v}{w}{x} = \mat{0}{\infty}{i}{0} \quad \text{or} \quad
\mat{a}{b}{c}{d} = \mat{i}{0}{0}{\infty}.$$
In particular, if $A(i)$ is written as the product of two matrices in
$M_{2}(\S)$, then one of those matrices is similar to $A(i)$.

A similar argument based on writing $B$ as the product of two elements completes
the proof.
\end{proof}

Since every generating set for $M_2(\Kmin)$ contains a matrix similar to $A(i)$
for every $i$, $M_2(\Kmin)$ is not finitely generated.

\begin{cor}\label{cor:MinPlusIrredundant}
  The generating sets given in \cref{thm:MinPlusGenSet} and
  \cref{cor:FiniteMinPlusGenSet} are irredundant.
\end{cor}
\begin{proof}
By \cref{lem:MinPlusSimilarGenerators}, any generating set for $M_{2}(\S)$,
where $\S \in \{\Kmin\} \cup \set{\Kmint}{t \in \Np}$, contains a matrix similar
to $A(i)$ for each $i \in \S$. Since the only such element in the given
generating set is $A(i)$ itself, it follows that the generators $A(i)$ are
irredundant. The same argument applies to $B$.

It is straightforward to show that if a matrix with $\infty$ occurring twice in
one row is expressed as a product of two matrices, then one of those matrices
has a row where $\infty$ occurs twice. Therefore, in order to  generate such
matrices, a generator with a row containing two occurrences of $\infty$ is
required. In the generating sets of \cref{thm:MinPlusGenSet} and
\cref{cor:FiniteMinPlusGenSet}, $C$ is the only such generator, so it is
irredundant.
\end{proof}

Since $M_2(\Kmin)$ is not finitely generated, it follows from the corollary that
the generating set for $M_2(\Kmin)$ given in \cref{thm:MinPlusGenSet} is
minimal. The proof of \cref{thm:MinPlusMinimality} is therefore complete in the
case of $\Kmin$.

\begin{lemma}\label{lem:MinPlusGensNonRelated}
  If $t \in \Np$ is arbitrary, then the generating set for $M_2(\Kmint)$ given
  in \cref{cor:FiniteMinPlusGenSet} consists of
  $\J$-non-equivalent elements.
\end{lemma}
\begin{proof}
  By \cref{prop:RowSpacesJRelation}, it is sufficient to
  demonstrate that the row spaces of the generators are non-isomorphic. In order
  to do so, we first observe several facts about spanning sets of these row
  spaces.

  It is routine to verify that the rows $(\infty~0)$, $(1~\infty)$, and
  $(0~\infty)$ are indecomposable, in the sense that every linear combination of
  rows equal to one of these rows must contain a non-zero scalar multiple of
  that row. It follows that $(\infty~0)$ is contained in any spanning set for
  $\RowS(C)$ and so $\{(\infty~0)\}$ is the unique minimal spanning set for
  $\RowS(C)$.  Similarly, $(\infty~0)$ and $(1~\infty)$ are contained in any
  spanning set for $\RowS(B)$, making $\{(\infty~0), (1~\infty)\}$ the unique
  minimal spanning set for $\RowS(B)$.

  To determine that $\RowS(A(i))$ has a unique minimal spanning set, we note
  that if $(x~0) \in \RowS(A(i))$, then we can write \[(x~0) = a(i~0) \oplus
    b(0~\infty) \] for some $a, b \in \Kmint$. From the second column, we have
  $a = 0$; hence $x \leq i$. If $(i~0)$ is written as a sum of rows in
  $\RowS(A(i))$, \[(i~0) = (c, d) \oplus (e~f),\] then one of $d$ or $f$ is $0$;
  without loss of generality we can assume $d = 0$.  By the previous
  observation, we have $c \leq i$, and by assumption $\min(c, e) = i$; hence $c
  = i$. It follows that $(i~0)$ is contained in any spanning set for
  $\RowS(A(i))$; hence the rows of $A$ form the unique minimal spanning set of
  $\RowS(A(i))$.  Since $\RowS(C)$ is spanned by the single row $(\infty~0)$
  while $\RowS(A(i))$ and $\RowS(B)$ cannot be spanned by a single element,
  $\RowS(C)$ is not isomorphic to $\RowS(A(i))$ or $\RowS(B)$. Any isomorphism
  from $\RowS(A(i))$ to $\RowS(B)$ must map the rows of $A(i)$ to the rows of
  $B$ as these form unique minimal spanning sets, but since the sets
  \[\set{a(i~0)}{a \in \Kmint} \COMMA
    \set{a(0~\infty)}{a \in \Kmint}
  \]
  both contain $t$ unique rows, while the sets 
  \[
    \set{a(1~\infty)}{a \in \Kmint}\COMMA
    \set{a(\infty~0)}{a \in \Kmint}
  \]
  contain $t - 1$ and $t$ rows respectively, no such isomorphism exists.

  It remains to show that $\RowS(A(i))$ and $\RowS(A(j))$ are non-isomorphic for
  $i \neq j$. Since \[(a~0) = 0(i~0) + a(0~\infty)\] for $a \in \Kmint$ such
  that $a \leq i$, together with the observation above we have $(a, 0) \in
  \RowS(A(i))$ if and only if $a \leq i$. It follows that there are precisely $i
  + 1$ rows $y \in \RowS(A(i))$ such that there exists $a \in \Kmint$ with $ay =
  (i~0)$. Similarly there is precisely one row $z \in \RowS(A(i))$ such that
  there exists $b \in \Kmint$ with $bz = (0~\infty)$.  Since any isomorphism
  from $\RowS(A(i))$ to $\RowS(A(j))$ must map the rows of $A(i)$ to the rows of
  $A(j)$, but there are $j + 1$ such rows for one of the rows of $A(j)$, there
  is no such isomorphism.
\end{proof}
This completes the proof of \cref{cor:FiniteMinPlusGenSet}.

For any $t \in \Np$, there is a quotient map $\phi_t: \Kmin \to \Kmint$ under
which $\phi_t(x) = \min(x, t)$ for all $x \in \Kmin\setminus\{\infty\}$ and
$\phi_t(\infty) = \infty$; this is a semiring homomorphism. It is routine to
verify that the map $\psi_t: M_n(\Kmin) \to M_n(\Kmint)$ defined by
entrywise application of $\phi_t$ is a semigroup homomorphism.
Let $A, B$ be two elements in the generating set of \cref{thm:MinPlusGenSet}.
Consider a finite quotient $\Kmint$ of $\Kmin$ to which all entries of
both $A$ and $B$ belong. The previous lemma shows that $A$ is not
$\J$-equivalent to $B$ in $M_n(\Kmint)$. But $\psi_t(A) = A$ and $\psi_t(B) = B$
and so $\psi_t(A)$ is not $\J$-equivalent to $\psi_t(B)$ in $M_n(\Kmint)$. It
follows that $A$ and $B$ are not $\J$-equivalent in $M_n(\Kmin)$. This
proves the following corollary.

\begin{cor}
  \label{cor:InfiniteMinPlusGensNonRelated}
  The generating set given in \cref{thm:MinPlusGenSet} consists of
  $\J$-non-equivalent elements.
\end{cor}


\subsection{Max-plus matrices}

The semirings $\Kmin$ and $\Kmax$ are in certain ways similar to each other; it
is only in the interaction of their respective additions and multiplication
operations that they display significantly different behaviour. The result of
this is that the process of proving that certain generating sets for
$M_2(\Kmax)$ and $M_2(\Kmaxt)$ are minimal is essentially the same as in the case
of $M_2(\Kmin)$ and $M_2(\Kmint)$; it only differs in some details.

\begin{thm}[\cite{East2020aa}]\label{thm:MaxPlusGenSet}
  The monoid $M_2(\Kmax)$ of $2 \times 2$ max-plus matrices is
  generated by the matrices 
  \[
    X(i) = \mat i00{-\infty} \COMMA
    Y = \mat 1{-\infty}{-\infty}0 \COMMA
    Z = \mat {-\infty}{-\infty}{-\infty}0 \COMMA
    W(j,k) = \mat 0jk0 ,
  \]
  where $i\in \Kmax$ and $j, k\in\N$ such that $1 \leq j \leq k$.
\end{thm}

\begin{cor}[\cite{East2020aa}]\label{cor:FiniteMaxPlusGenSet}
  Let $t\in \Np$ be arbitrary. Then the finite monoid $M_2(\Kmaxt)$ of $2
  \times 2$ max-plus matrices is generated by the $(t^2+3t+8)/2$ matrices 
  \[
    X(i) = \mat i00{-\infty} \COMMA
  Y = \mat 1{-\infty}{-\infty}0 \COMMA
  Z = \mat {-\infty}{-\infty}{-\infty}0 \COMMA
  W(j,k) = \mat 0jk0 ,
  \]
  where $i\in\Kmaxt$ and $j,k\in\{1,\ldots,t\}$ with $j\leq k$.
\end{cor}

\begin{thm}\label{thm:MaxPlusMinimality}
  The generating sets of \cref{thm:MaxPlusGenSet} and
  \cref{cor:FiniteMaxPlusGenSet} are irredundant and minimal.
\end{thm}
The proof of this theorem follows the same pattern as the proof of
\cref{thm:MinPlusMinimality}.

\begin{lemma}\label{lem:MaxPlusSimilarGenerators}
  Let $\S \in \{\Kmax\} \cup \set{\Kmaxt}{t \in \Np}$. Then any generating
  set for $M_{2}(\S)$ contains matrices similar to each matrix other than $Z$
  in the appropriate generating set from \cref{thm:MaxPlusGenSet} or \cref{cor:FiniteMaxPlusGenSet}.
\end{lemma}
\begin{proof}
  As in the proof of \cref{lem:MinPlusSimilarGenerators}, it suffices to show
  that if any of the generators described is expressed as a product of two
  matrices, then one of those factors must be similar to the generator. We will
  only prove this for $W(j, k)$ as an illustrative example; the other proofs are
  similar.

  Writing 
  \[W(j, k) = \mat{0}{j}{k}{0} = \mat{a}{b}{c}{d} \mat{u}{v}{w}{x},\]
  for some $a,b,c,d,u,v,w,x \in \S$, we obtain the conditions
  \begin{enumerate}
    \item
      ($a = u = 0$ and $b + w \leq 0$) or ($b = w = 0$ and $a + u \leq 0$),
    \item
      ($a + v = j$ and $b + x \leq j$) or ($b + x = j$ and $a + v \leq j$),
    \item
      ($c + u = k$ and $d + w \leq k$) or ($d + w = k$ and $c + u \leq k$),
    \item
      ($c = v = 0$ and $d + x \leq 0$) or ($d = x = 0$ and $c + v \leq 0$).
  \end{enumerate}
  First assume that in $(1)$, we have $a = u = 0$. If we also assume that in
  $(4)$ we have $c = v = 0$, then from $(2)$ we have $b + x = j$ and hence both
  $b$ and $x$ are not equal to $-\infty$. It follows from $(1)$ that $b = w = 0$
  or $w = -\infty$. In the latter case, $(3)$ cannot be satisfied. In the
  former, it follows from $(3)$ that $d = k$. However, then $(4)$ cannot be
  satisfied as $d + x > 0$ unless $x = -\infty$, which we noted above was not
  the case. Hence in $(4)$, we must have $d = x = 0$. We now consider three
  cases:
  \begin{enumerate}
    \item
      If $w > 0$, then by $(1)$ we have $b = -\infty$. Hence, by $(2)$, $v = j$
      since $a = 0$. It follows from $(4)$ that $c = -\infty$ and thus by $(3)$
      $w = k$, and the second factor is similar to $W(j, k)$.
    \item
      If $w = -\infty$, then from $(3)$ we have $c = k$, since $u = 0$. It
      follows that $v = -\infty$ (by $(4)$), and thus that $b = j$ (by $(2)$).
      The first factor is then similar to $W(j, k)$.
    \item
      If $w = 0$, then by $(1)$ we have $b \leq 0$. It follows from $(2)$ that
      $v = j$, and hence from $(4)$ that $c = -\infty$. However, it is then
      impossible to satisfy $(3)$, and so this case cannot arise.
  \end{enumerate}
  In each case that may arise, one factor is similar to $W(i, j)$. It now
  remains to consider the case of $b = w = 0$ in $(1)$; the proof of this case
  is dual to argument above.
\end{proof}

Since every generating set for $M_2(\Kmax)$ contains a matrix similar to $X(i)$
for every $i$, $M_2(\Kmax)$ is not finitely generated.

Note that the previous lemma fails for $Z$, since 
\[Z = \mat{-\infty}{-\infty}{-\infty}{0} = \mat{-\infty}{-\infty}{0}{0}
  \mat{-\infty}{0}{-\infty}{0}.\]
In order to show that $Z$ is an irredundant generator we use the following
lemma, the proof of which is elementary.

\begin{lemma}\label{lem:MaxPlusTwoInfsInRow}
  Let $\S \in \{\Kmax\} \cup \set{\Kmaxt}{t \in \Np}$. If a matrix $A \in M_2(\S)$
  which has a row in which both entries are $-\infty$ is written as a
  product $A = BC$ of two matrices $B, C \in M_2(R)$, then one of $B$ or $C$ has
  a row in which both two entries are $-\infty$.
\end{lemma}

\begin{cor}\label{cor:MaxPlusIrredundant}
  The generating sets given in \cref{thm:MaxPlusGenSet} and
  \cref{cor:FiniteMaxPlusGenSet} are irredundant.
\end{cor}
\begin{proof}
  This follows from \cref{lem:MaxPlusSimilarGenerators} and
  \cref{lem:MaxPlusTwoInfsInRow} in the same way as in
  \cref{cor:MinPlusIrredundant}.
\end{proof}

As in the previous section, it follows from the corollary that
the generating set for $M_2(\Kmax)$ given in \cref{thm:MaxPlusGenSet} is
minimal.

In order to demonstrate that the generating sets of
\cref{thm:MaxPlusGenSet} consist of $\J$-non-equivalent elements, we establish
the following strong result on row bases of max-plus matrices.
\begin{lemma}
  \label{lem:MaxPlusUniqueRowBasis}
  For any $\S \in \{\Kmax\} \cup \set{\Kmaxt}{t \in \Np}$ and matrix $A \in
  M_n(\S)$, $A$ has a unique row-basis consisting of the set of those rows
  $A_{i*}$ not expressible as a linear combination of rows distinct from
  $A_{i*}$.
\end{lemma}
\begin{proof}
We first note that $\oplus$ and $\otimes$ by an element in
$R\setminus\{-\infty\}$ preserves or increases the order of any element in $\S$.
It follows that if $aA_{i*}$ is a term in some linear combination equal to a row
$x$, for some $a \in \S\setminus\{-\infty\}$, then $x_j \geq A_{ij}$ for $1 \leq
j \leq n$. Let $W$ be those rows of $A$ distinct from $A_{i*}$, and let
$\genset{W}$ denote the row space spanned by these rows. Suppose that $A_{i*}
\not\in \genset{W}$. Then for any row $x \in \RowS(A)\setminus \genset{W}$, $x$
can be expressed as a linear combination with a non-trivial term involving
$A_{i*}$; hence $A_{ij} \leq x_j$ for $1 \leq j \leq n$. It follows that
$A_{i*}$ cannot be written as a linear combination involving $x$ in a
non-trivial way unless $x = A_{i*}$. Thus $A_{i*} \not\in \genset{\RowS(A)
  \setminus \{ A_{i*} \}}$. It follows that any row $A_{i*}$ which cannot be
written as a linear combination of distinct other rows of $A$ must be included
in any spanning set for $\RowS(A)$. We must also show that such rows do span
$\RowS(A)$.

If $X$ is any set of rows of $A$ and $A_{x*} \in X$ is such that $A_{x*} \in
\genset{X \setminus\{A_{x*}\}}$, then as usual $\genset{X} = \genset{X \setminus
  \{A_{x*}\}}$. We must establish that if we also have $A_{y*} \in X\setminus\{A_{x*}\}$
such that $A_{y*} \in \genset{X\setminus\{A_{y*}\}}$, then $A_{y*} \in
X\setminus\{A_{x*}, A_{y*}\}$; the result then follows by repeatedly removing
such rows from $A$. 

Suppose that some term $aA_{y*}$ appears in a linear combination equal
to $A_{x*}$, where $a \in \S\setminus\{-\infty\}$. This implies that $A_{y*} \leq
A_{x*}$ in every component, and for some component the inequality is strict since
$A_{y*} \neq A_{x*}$. It follows that $A_{x*}$ cannot be involved in a
non-trivial way in any linear combination equal to $A_{y*}$, and hence $A_{y*}
\in \genset{X \setminus\{A_{x*}, A_{y*}\}}$. Suppose instead that no such term
$aA_{y*}$ exists; that is, $A_{x*} \in \genset{X \setminus\{A_{x*}, A_{y*}\}}$.
Then the same linear combination of elements of $X\setminus\{A_{y*}\}$ that is
equal to $A_{y*}$, possibly with $A_{x*}$ replaced by the linear
combination of elements in $X \setminus\{A_{x*}, A_{y*}\}$ that is equal to
$A_{x*}$, witnesses that $A_{y*} \in X\setminus\{A_{x*}, A_{y*}\}$.
\end{proof}

\begin{lemma}\label{lem:MaxPlusGensNonRelated}
  The generating sets given in \cref{cor:FiniteMaxPlusGenSet} consist of
  $\J$-non-equivalent elements.
\end{lemma}
\begin{proof}
  Fix $t \in \Np$. By \cref{prop:RowSpacesJRelation}, it suffices to show that the
generating sets consist of elements with non-isomorphic row spaces. By
\cref{lem:MaxPlusUniqueRowBasis}, each row space has a unique basis; it
follows that any isomorphism between these row spaces must map one unique basis
to another. We calculate the number of distinct scalar multiples of each basis
element, which is an isomorphism invariant:

\begin{itemize}
  \item 
    $\RowS(Y)$ has unique basis $\{(1~\infty), (-\infty~0)\}$, with $t + 1$ and
    $t + 2$ distinct scalar multiples (including themselves) respectively;
  \item
    $\RowS(Z)$ has unique basis $\{(-\infty~0)\}$, with $t + 2$ distinct scalar
    multiples;
  \item
    $\RowS(X(i))$, for any $i \in \Kmaxt$, has unique basis
    $\{(i~0),~(0~-\infty)\}$, with $t + 2$ distinct scalar multiples each;
  \item $\RowS(W(j, k))$, for $1 \leq j \leq k \leq t$, has unique basis
    $\{(0~j),~(k~0)\}$, with $t + 2$ distinct scalar multiples each.
\end{itemize}
It follows that the only possible isomorphisms between row spaces of distinct
generators are between $\RowS(W(j, k))$, $\RowS(W(p, q))$, $\RowS(X(i))$, and
$\RowS(X(l))$ for some $(p, q) \neq (j, k)$ and $l \neq i$. We note that in
$\RowS(X(i)$ the sum of the two basis rows is equal to one of the basis rows,
while the same is not true in $\RowS(W(j, k))$; hence there can be no
isomorphism between row spaces of these types of generators. Furthermore,
$\RowS(X(i))$ is non-isomorphic to $\RowS(X(j))$ for $i \neq j$ since $i$ is the
largest value such that for some choice of $g_1$ and $g_2$ being distinct basis
rows, $ig_1 \oplus g_2 = g_2$, and this is an isomorphism invariant.

We show that $\RowS(W(j, k))$ is not isomorphic to $\RowS(W(p, q))$ in a similar
way. Let $x_1, x_2$ be distinct basis rows of $W(j, k)$. Then there exist least
elements $a, b \in \{1, \ldots, t\}$ such that $ax_1 \oplus x_2 = ax_1$ and
$bx_2 \oplus x_1 = bx_2$. By examination of the basis given above, the lesser of
$a$ and $b$ (if there is one) is equal to $j$ and the greater is equal to $k$
(or all of the values are equal). Since $\{a, b\}$ is an isomorphism invariant,
there is no isomorphism between $\RowS(W(j, k))$ and $\RowS(W(p, q))$ for $(j,
k) \neq (p, q)$.
\end{proof}
By \cref{prop:Wilf}, this completes the proof of
\cref{cor:FiniteMinPlusGenSet}. As in the previous section, a corollary of the
previous lemma is that the generating set given in \cref{thm:MinPlusGenSet}
also consists of $\J$-non-equivalent elements.


\section{Matrices over $\mathbb{Z}_n$}
\label{sec:IntegersModN}
\setcounter{subsection}{1}

Perhaps surprisingly, there appears to be no known general characterisation of a
minimal generating set of the multiplicative group $U_n$ of integers coprime to
$n$, modulo $n$. Any description of a minimal generating set for $M_k(\Z_n)$ for
general $k$ would provide such a minimal generating set for $U_n$, by taking $k
= 1$ and restricting the generators to those of the group of units. This leads
us to consider the so-called \defn{relative rank} of $M_k(\Z_n)$ with respect to
the group of units $GL_k(\Z_n)$. For a semigroup $S$ and subset $A \subseteq S$,
the relative rank of $S$ with respect to $A$, denoted $\d(S:A)$, is defined to
be the minimum cardinality of any subset $B$ such that $A \cup B$ generates $S$.
In this section, we will determine $\d(M_k(\Z_n):GL_k(\Z_n))$ by proving the
following analogues to Devadze's Theorem (see \cref{thm:Devadzefull}).

\begin{thm}
  \label{thm:ZnGenSet}
    Let $k,n \in \Np$.
    The monoid $M_k(\Z_n)$ is generated by any set consisting of $GL_k(\Z_n)$
    and one element from each of the $\J$-classes
    immediately below $GL_k(\Z_n)$ in the $\J$-order of $M_k(\Z_n)$.
    For example,
    $M_k(\Z_n)$ is generated by the set
    $$
    GL_k(\Z_n)
    \cup
    \left\{
    \begin{pmatrix}
      p \operatorname{mod} n & 0      & \cdots & 0      \\
      0       & 1      &        & \vdots \\
      \vdots  &        & \ddots & 0      \\
      0       & \cdots & 0      & 1      \\
     \end{pmatrix}
     \ :\
     p\ \text{is a prime divisor of}\ n\ \text{in}\ \mathbb{N}
     \right\}.
     $$
\end{thm}

For each prime divisor $p$ of $n$ in $\N$, we write
  $$X_p =
    \begin{pmatrix}
      p \operatorname{mod} n & 0      & \cdots & 0      \\
      0       & 1      &        & \vdots \\
      \vdots  &        & \ddots & 0      \\
      0       & \cdots & 0      & 1      \\
    \end{pmatrix}$$
and denote by $\mathcal{X}_n$ the set $\set{X_p}{p \text{ is a prime divisor of
  } n}$.

Note that the entry $p$ mod $n$ is, of course, simply $p$ unless $p = n$. Since
the generating sets defined in \cref{thm:ZnGenSet} contain precisely one
element from each $\J$-class immediately below the group of units, those
elements are irredundant. By \cref{prop:Wilf}, this proves the following
theorem:
\begin{thm}
  \label{thm:ZnMinimal}
  Let $k,n \in \Np$, let $G$ be a generating set of minimum cardinality for
  $GL_k(\Z_n)$, and let $P$ be a set of representatives of the
  $\J$-classes immediately below $GL_k(\Z_n)$ in the $\J$-order of
  $M_k(\Z_n)$. Then $G \cup P$ is a minimal generating set for $M_k(\Z_n)$.
  In particular,
  $$\d(M_k(\Z_n) : GL_k(\Z_n))
                       =
                       |\set{p}{p\ \text{is a prime divisor of}\ n}|.
  $$
\end{thm}

We note that \cref{thm:ZnGenSet} and \cref{thm:ZnMinimal} hold trivially for
$n = 1$ as $1$ has no prime divisors and for any $k \in \Np$, $|M_k(\Z_1)| = 1$.
The further results and proofs in this section may also be stated in such a way
as to hold for $n=1$, but this is not necessary and results in slightly less
tidy statements.

In order to prove \cref{thm:ZnGenSet}, we will show that each $\J$-class of
$M_k(\Z_n)$ contains a matrix in a certain standard form; we will then show how
to generate these standard forms and hence generate the full matrix monoid.

In order to demonstrate that these standard forms exist in $M_k(\Z_n)$, we first
prove the following elementary but useful lemma, without any claim of
originality.

\begin{lemma}
  \label{lem:ModularProducts}
  Let $a \in \Z_n \setminus U_n$, with $gcd(a, n) = d > 1$. Then there exists $b
  \in U_n$ such that $ab = d$ mod $n$.
\end{lemma}
\begin{proof}
  By Bezout's identity, there exist some $x, y \in \Z$ such that $ax + ny = d$;
  equivalently $\frac{a}{d}x + \frac{n}{d}y = 1$. It follows that
  $\frac{n}{d}$ and $x$ are coprime. For every $k \in \Z$, we also have
  $\frac{a}{d}(x + k\frac{n}{d}) + \frac{n}{d}(y - k\frac{a}{d}) = 1$. It is
  therefore sufficient to find a value of $x + k\frac{n}{d}$ which is coprime to
  $n$. Since $x$ and $\frac{n}{d}$ are coprime, by setting $k$ to be the product
  of those prime factors of $n$ that do not appear in $x$ or in $\frac{n}{d}$ we
  have that each prime factor of $n$ appears in precisely one of the terms $x$
  and $k\frac{n}{d}$; hence no prime factor of $n$ divides the sum $x +
  k\frac{n}{d}$.
\end{proof}

\begin{lemma}
  \label{lem:Wilf1}
  For $n > 1$, every matrix $Y \in M_k(\Z_n)$ can be written as a product $AXB$,
  for two matrices $A, B \in GL_k(\Z_n)$ and a diagonal matrix $X \in J_Y$ of
  the form
  $$\begin{pmatrix}
      0      &        & \cdots &        &        & 0       \\
             & \ddots &        &        &        &         \\
      \vdots &        & 0      &        &        & \vdots  \\
             &        &        & d_1    &        &         \\
             &        &        &        & \ddots & 0       \\
      0      &        & \cdots &        & 0      & d_m     \\
    \end{pmatrix},$$
    for some $d_1, \ldots, d_m \in \Z_n \setminus \{0\}$,
    where $d_1, \ldots, d_m | n$ in $\mathbb{N}$,
    and $d_1 \geq \cdots \geq d_m$.
\end{lemma}

\begin{proof}
  It is straightforward to show that there are units in $GL_k(\Z_n)$ which, via
  left multiplication, apply the elementary row operations of adding one row to
  another, subtracting one row from another, or exchanging rows in any matrix $Y
  \in M_k(\Z_n)$. This is familiar from linear algebra; the only ``missing''
  standard row operation is scaling by an arbitrary scalar, which is not
  always represented by an element of $GL_k(\Z_n)$. However, scaling some row by
  an element of the group $U_n$ does still correspond to left multiplication by
  an element of $GL_k(\Z_n)$. The dual statements also hold for column operations.
  Fix $Y \in M_K(\Z_n)$. It is sufficient to show that elementary row and column operations
  can ``diagonalise'' $Y$ to obtain $X$; this implies that $Y = AXB$ for $A,
  B \in GL_k(\Z_n)$, and, since multiplication by units preserves $\J$-classes,
  that $X \in J_Y$.
  The diagonalisation procedure applied is essentially that of Gaussian
  elimination, but it is not altogether obvious that such a process applies in
  this case. This is what we now demonstrate.

  Suppose that the entry $Y_{ij}$ is non-zero. Define $g$ to be the greatest
  common divisor of the entries of $Y_{i*}$. By repeatedly subtracting one
  column from another to obtain a new matrix $Y'$, we may apply the Euclidean
  algorithm so that some entry $Y'_{ik}$ in the $i$th row is equal to $g$. By
  subtracting column $k$ from column $j$ an appropriate number of times, we may
  assume that $Y'_{ij} = g$. Similarly, if $h$ is defined as the greatest common
  divisor of $g$ and the entries of $Y_{*j}$, by repeatedly subtracting
  appropriate rows we may produce a matrix $Y''$ with $Y''_{ij} = h$. Since $h$
  divides every entry in row $i$ and column $j$ of $Y''$, we may then use row and
  column subtractions to clear each non-zero entry in row $i$ and column $j$ of
  $Y''$ other than $Y''_{ij}$. Repeating this process produces a matrix $V$ with
  at most one non-zero entry in each row and column. For each row of $V$ with a
  non-zero entry $x$ coprime to $n$, we may scale that row by the inverse of $x$
  in $U_n$ by left multiplication by some unit; this produces a matrix $W$ with
  each entry either $1$ or non-coprime with $n$. By \cref{lem:ModularProducts},
  if row $i$ of $W$ contains the non-zero entry $a$ where $gcd(a, n) = d > 1$,
  there exists some $b \in U_n$ with $ab = d$ mod $n$. We can therefore
  left-multiply by some unit to scale row $i$ so that the non-zero entry in row
  $i$ of the product is $d$. By repeating this for each such row, and then
  permuting rows and columns, the matrix $X$ with the required diagonal form is
  obtained.
\end{proof}

We say that matrices in the form of \cref{lem:Wilf1} are in ``standard diagonal
form''.  The row space $\RowS(X)$ of such a standard diagonal form matrix can
easily be seen to be isomorphic (as an additive group) to the direct product
$G(X) = \Z_{n/d_1} \times \cdots \times \Z_{n/d_m}$. Note that for any prime
divisor $p$ of $n$, $X_p$ is in the form of \cref{lem:Wilf1}. Furthermore, by
considering the prime power decomoposition of $G(X_p)$, it is straightforward to
determine that there are no other matrices $X$ of the same form with $G(X) \cong
G(X_p)$. Hence there are no other matrices $X$ of the same form with $\RowS(X)
\cong \RowS(X_p)$. By \cref{prop:RowSpacesJRelation}, this observation yields
the following lemma. 

\begin{lemma}
  \label{lem:XpUnique}
  Let $p$ be a prime divisor of $n$. Then $X_p$ is the unique standard diagonal
  form matrix in $J_{X_p}$.
\end{lemma}

\begin{cor}
  \label{cor:XpDistinct}
  Let $p_1$ and $p_2$ be distinct prime divisors of $n$. Then $X_{p_1}$
  and $X_{p_2}$ are not $\J$-related.
\end{cor}

In the following two lemmas, we show that the example set
given in \cref{thm:ZnGenSet} indeed generates $M_k(\Z_n)$.

\begin{lemma}
  \label{lem:Wilf2}
  Let $p$ be a prime divisor of $n$.
  Then $J_{X_p} \subseteq \genset{GL_k(\Z_n) \cup \{Y\}}$ for all $Y \in
  J_{X_p}$.
\end{lemma}

\begin{proof}
    Let $Y \in J_{X_{p}}$.
    By \cref{lem:Wilf1,lem:XpUnique},
    $J_{X_{p}} = \set{AX_{p}B}{A,B \in GL_k(\Z_n)}$;
    in particular, $Y = EX_{p}F$ for some $E,F \in
    GL_k(\Z_n)$. Hence 
    \begin{align*}
      J_{X_p} &= \set{A(E^{-1}YF^{-1})B}{A,B \in GL_k(\Z_n)} \\
              &\subseteq \genset{GL_k(\Z_n) \cup \{Y\}}. \qedhere
    \end{align*}
\end{proof}

\begin{lemma}\label{lem:ZnGenSet}
  Let $\mathcal{Y} \subseteq M_k(\Z_n)$ be a set of representatives of the
  \(\J\)-classes \(\set{J_{X_{p}}}{p \text{ is a prime divisor of } n }\).
  Then $M_k(\Z_n) = \genset{GL_k(\Z_n) \cup \mathcal{Y}}$.
\end{lemma}

\begin{proof}
  By \cref{lem:Wilf2}, $\mathcal{X}_n \subseteq \genset{GL_k(\Z_n) \cup
    \mathcal{Y}}$.
  Since the matrix
  \[
  \begin{pmatrix}
      0       & 0      & \cdots & 0      \\
      0       & 1      &        & \vdots \\
      \vdots  &        & \ddots & 0      \\
      0       & \cdots & 0      & 1      \\
  \end{pmatrix}
\]
  is a product of matrices in $\mathcal{X}_n$, and row- and column-permutations
  are represented by left and right multiplication by units, $\genset{GL_k(\Z_n)
    \cup \mathcal{Y}}$ contains every diagonal matrix that differs from
  the identity matrix by having a $0$ or a prime divisor of $n$ in exactly one
  place on the main diagonal. It is straightforward to see that any matrix in
  standard diagonal form may be expressed as a product of such matrices. The
  result now follows by \cref{lem:Wilf1}.
\end{proof}

In order to complete the proof of \cref{thm:ZnGenSet}, it remains to show that
the $\J$-classes $J_{X_p}$ lie immediately below the group of units for all $X_p
\in \mathcal{X}_n$. Since any generating set for $M_k(\Z_n)$ must contain a
representative of each $\J$-class immediately below the group of units, the set
$\mathcal{X}_n$ must contain a representative of each $\J$-class immediately
below the group of units. It only remains to prove that all elements of
$\mathcal{X}_n$ lie in such a maximal $\J$-class.

\begin{lemma}
  \label{lem:ZnImmediatelyBelow}
  The $\J$-classes $J_{X_p}$ are precisely the $\J$-classes immediately below
  the group of units in the $\J$-order on $M_k(\Z_n)$.
\end{lemma}
\begin{proof}
Observe that as an additive group, $\RowS(X_p)$ is a maximal subgroup of
$\RowS(I)$, where $I$ is the identity matrix. Let $J$ be a $\J$-class with
$J_{X_p} \leq J \leq GL_k(\Z_n)$, and suppose that $J$ lies immediately below
$GL_k(\Z_n)$. We wish to show that $J = J_{X_p}$. By \cref{lem:ZnGenSet} $J \in
\set{J_{X_q}}{X_q \in \mathcal{X}_n}$ as otherwise $J$ would not be generated by
$\mathcal{X}_n \cup GL_k(\Z_n)$. Let $q$ be the prime divisor of $n$ such that
$J = J_{X_q}$. It follows that $X_p = AX_qB$ for some matrices $A, B \in
M_k(\Z_n)$. The row space $\RowS(AX_qB) = \RowS(X_p)$ is an additive subgroup of
$\RowS(X_qB)$; since $\RowS(X_p)$ is a maximal subgroup, and $|\RowS(X_qB)| <
\RowS(I)$, it follows that $\RowS(X_p) = \RowS(X_qB)$. Right multiplication by
$B$ is therefore a surjective homomorphism from $\RowS(X_q)$ to $\RowS(X_p)$,
but by Lagrange's Theorem no such homomorphism exists for $p \neq q$; hence $J =
J_p$.
\end{proof}

\printbibliography
\end{document}